\documentclass[final,11pt]{article}
\usepackage[english]{babel}
\usepackage{amsthm}
\usepackage{amsmath,amssymb,amsfonts}
\usepackage[dvipsnames,usenames]{color}
 
\usepackage[bitstream-charter]{mathdesign}
\usepackage[utf8]{inputenc}
\usepackage{tikz}
\usepackage{caption}
\usepackage{subcaption}
\usepackage{enumerate}
\usepackage{geometry}
\usepackage[affil-it]{authblk}
\usepackage{tikz-cd}
\usepackage{titling}

\newtheorem{thm}{Theorem}[section]
\newtheorem{prop}[thm]{Proposition}
\newtheorem{lem}[thm]{Lemma}
\newtheorem{defi}[thm]{Definition}

\newtheorem{prob}{Problem}
\newtheorem{cor}[thm]{Corollary}
\newtheorem{claim}[thm]{Claim}
\newtheorem{rem}[thm]{Remark}
\newtheorem{conj}{Conjecture}

\newcommand{\conv}{\ensuremath{{\rm conv}}}
\newcommand{\covectors}{\ensuremath{\mathcal{L}}}
\newcommand{\topegraphs}{\ensuremath{\mathcal{G}_{\rm COM}}}
\newcommand{\topegraphsOM}{\ensuremath{\mathcal{G}_{\rm OM}}}
\newcommand{\topegraphsAOM}{\ensuremath{\mathcal{G}_{\rm AOM}}}
\newcommand{\topegraphsLOP}{\ensuremath{\mathcal{G}_{\rm LOP}}}
\newcommand{\AG}{\ensuremath{\mathrm{AG}}}

\newcommand{\Q}{\ensuremath{\mathcal{Q}^-}}

\newcommand{\topes}{\ensuremath{\mathcal{T}}}

\makeatletter
\DeclareRobustCommand{\em}{%
  \@nomath\em \if b\expandafter\@car\f@series\@nil
  \normalfont \else \bfseries\itshape \fi}
\makeatother

\begin{document}
\thanksmarkseries{arabic}
\title{On tope graphs of complexes of oriented matroids}
\author{Kolja Knauer\thanks{Electronic address: \texttt{kolja.knauer@lis-lab.fr}} \and Tilen Marc\thanks{Electronic address: \texttt{tilen.marc@imfm.si}}}

\maketitle

\begin{abstract} 
  We give two graph theoretical characterizations of tope graphs of (complexes of) oriented matroids. The first is in terms of excluded partial cube minors, the second is that all antipodal subgraphs are gated. A direct consequence is a third characterization in terms of zone graphs of tope graphs. 
  
  Further corollaries include a characterization of topes of oriented matroids due to da Silva, another one of Handa, a characterization of lopsided systems due to Lawrence, and an intrinsic characterization of tope graphs of affine oriented matroids. Moreover, we obtain purely graph theoretic polynomial time recognition algorithms for tope graphs of the above and a finite list of excluded partial cube minors for the bounded rank case. 
  
  In particular, our results answer a relatively long-standing open question in oriented matroids and can be seen as identifying the theory of (complexes of) oriented matroids as a part of metric graph theory. Another consequence is that all finite Pasch graphs are tope graphs of complexes of oriented matroids, which confirms a conjecture of Chepoi and the two authors.
\end{abstract}

\tableofcontents

\section{Introduction}

%



A graph $G=(V,E)$ is a \emph{partial cube} if it is (isomorphic to) an isometric subgraph of a hypercube graph $Q_n$,
i.e., $d_G(u,v)=d_{Q_n}(u,v)$ for all $u,v\in V$, where $d$ denotes the distance function of the respective graphs.
Partial cubes were introduced by Graham and Pollak~\cite{Gra-71} in the study of interconnection networks. They form an important graph class in media theory~\cite{Epp-08}, frequently appear in chemical graph theory~\cite{Epp-09}, and quoting~\cite{Kla-12} present one of the central and most studied classes in metric graph theory.

Important subclasses of partial cubes include median graphs, bipartite cellular graphs, hypercellular graphs, Pasch graphs, and netlike partial cubes. 
Partial cubes also capture several important graph classes not directly coming from metric graph theory, such as region graphs of hyperplane arrangements, diagrams of distributive lattices, linear extension graphs of posets, tope graphs of oriented matroids (OMs), tope graphs of affine oriented matroids (AOMs), and lopsided systems (LOPs). A recently introduced unifying generalization of these classes are complexes of oriented matroids (COMs), whose tope graphs are partial cubes as well~\cite{Ban-15}. As it turns out, all of the above mentioned classes or partial cubes are indeed tope graphs of COMs.

Partial cubes admit a natural minor-relation (pc-minors for short) and several of the above classes including tope graphs of COMs are pc-minor closed. Complete (finite) lists of excluded pc-minors are known for median graphs, bipartite cellular graphs, hypercellular graphs and Pasch graphs, see~\cite{Che-86,Che-94,Che-16}. Another well-known construction of a smaller graph from a partial cube is the zone graph~\cite{Kla-12}.

In this paper we focus on COMs and their tope graphs. We present two  characterizations of the tope graphs and thus two graph theoretical characterizations of COMs. The first characterization is in terms of its complete (infinite) list of excluded pc-minors. 
As corollaries we obtain excluded pc-minor characterizations for tope graphs of OMs, AOMs, and LOPs. Moreover, in the case of bounded rank the list of excluded pc-minors is finite. We devise a polynomial time algorithm for checking if a given partial cube has another one as pc-minor, leading to polynomial time recognition algorithms for the classes with a finite list of excluded pc-minors. Another consequence is a characterization of tope graphs of COMs in terms of iterated zone graphs, which generalizes a result of Handa~\cite{ha-90} about tope sets of OMs.

The second characterization of tope graphs of COMs is in terms of the metric behavior of certain subgraphs. More precisely, we prove that a partial cube is the tope graph of a COM if and only if all of its antipodal (also known as symmetric-even~\cite{Ber-88}) subgraphs  are gated. As corollaries, this theorem specializes to tope graphs of OMs, AOMs, and LOPs. In particular, we obtain a new unified proof for characterization theorems of tope sets of LOPs and OMs due to Lawrence~\cite{Law-83} and da Silva~\cite{daS-95}, respectively. Moreover, this characterization allows to prove that Pasch graphs are COMs, confirming a conjecture of Chepoi, Knauer, and Marc~\cite{Che-16}. Finally, our characterization is verifiable in polynomial time, hence gives polynomial time recognition algorithms for tope graphs of COMs, OMs, AOMs, and LOPs, even without bounding the rank. Note that a polynomial time recognition algorithm for tope graphs of OMs was known before, see~\cite{Fuk-91}. However, this algorithm works without a characterization of the graphs, but constructs the set of cocircuits from the topes and there verifies the cocircuit axioms.   

In particular, we answer a long-standing open question on OMs, i.e., the question for a purely graph theoretical characterization of tope graphs, see~\cite[Problem 2]{Han-93} that can furthermore be verified in polynomial time, which was posed in~\cite[Problem 1.2]{Fuk-04}. Since the tope graph determines a COM, OM, AOM, or LOP up to isomorphism, see~\cite{Ban-15}, our results can be seen as identifying the theory of (complexes of) oriented matroids as a part of metric graph theory.

\begin{figure}[htb]
\centering
\includegraphics[width=\textwidth]{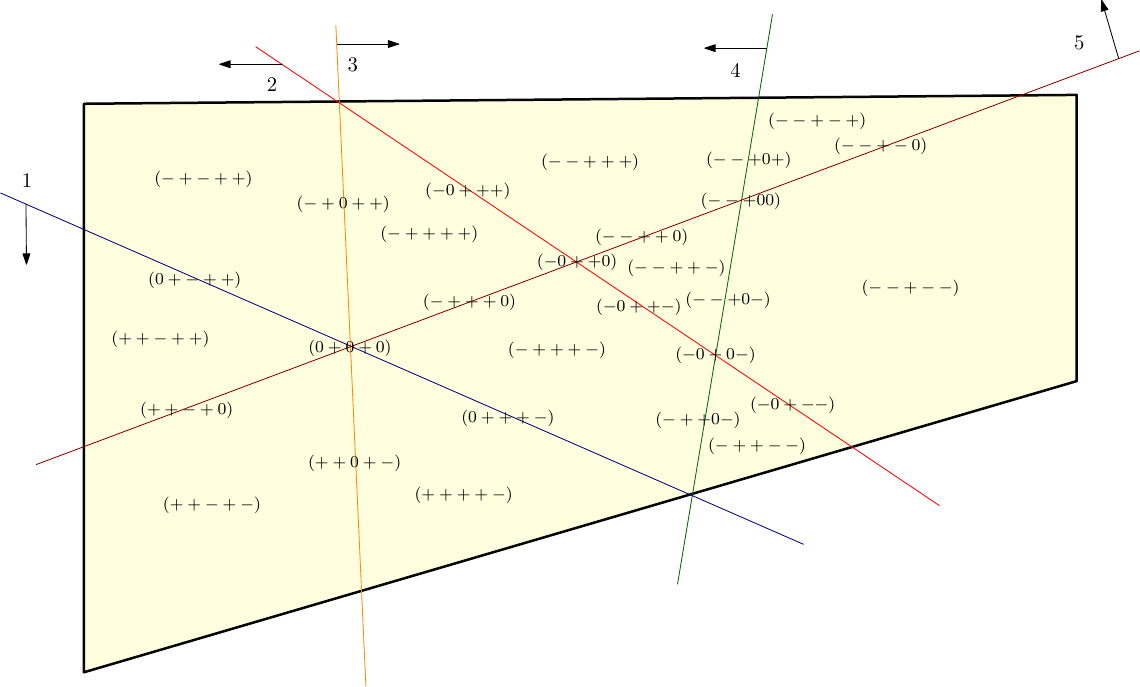}
\caption{A COM realized by a hyperplane arrangement and an open polyhedron in $\mathbb{R}^2$. The arrows on the hyperplanes indicate their positive side.}
\label{fig:COMexample}
\end{figure}

\bigskip

\noindent\textbf{Content of the paper:} 
The main theorem of the paper is Theorem~\ref{thm:main}, saying that a graph $G$ is the tope graph of a COM, i.e., $G\in \topegraphs$, if and only if $G$ is antipodally gated, i.e., $G\in\AG$, if and only if $G$ does not contain a partial cube minor from a specific set $\Q$, i.e., $G \in \mathcal{F}(\Q)$. (Generally, for a set ${X}$ of partial cubes we denote by $\mathcal{F}({X})$ the class of partial cubes that do not have any graph from $X$ as a partial cube minor.) Let us here give an outline of its proof ingredients and its corollaries, before going into the technicalities starting in the next section. The general strategy will be to prove first $\topegraphs\subseteq \AG$, then $\mathcal{F}(\Q)\subseteq \topegraphs$, and finally $\AG\subseteq \mathcal{F}(\Q)$.

Let us start by introducing the idea of a COM, which usually is encoded as a system of sign-vectors $\covectors\in\{+,-,0\}^{\mathcal{E}}$ on a finite ground set $\mathcal{E}$. In the illustrative {realizable} case one can think of $\covectors$ as the relative positions of cells in an oriented hyperplane arrangement intersected with an open polyhedron. See Figure~\ref{fig:COMexample} for an instructive example. This is a unifying generalization of realizable LOPs, OMs, and AOMs. The definition of LOPs, COMs, OMs, and AOMs can all be given in terms of axiomatics of the set of covectors $\covectors$ and also from this point of view it is reflected how COMs generalize the other ones naturally, see Definition~\ref{def:main}.

\begin{figure}[htb]
\centering
\includegraphics[width=.7\textwidth]{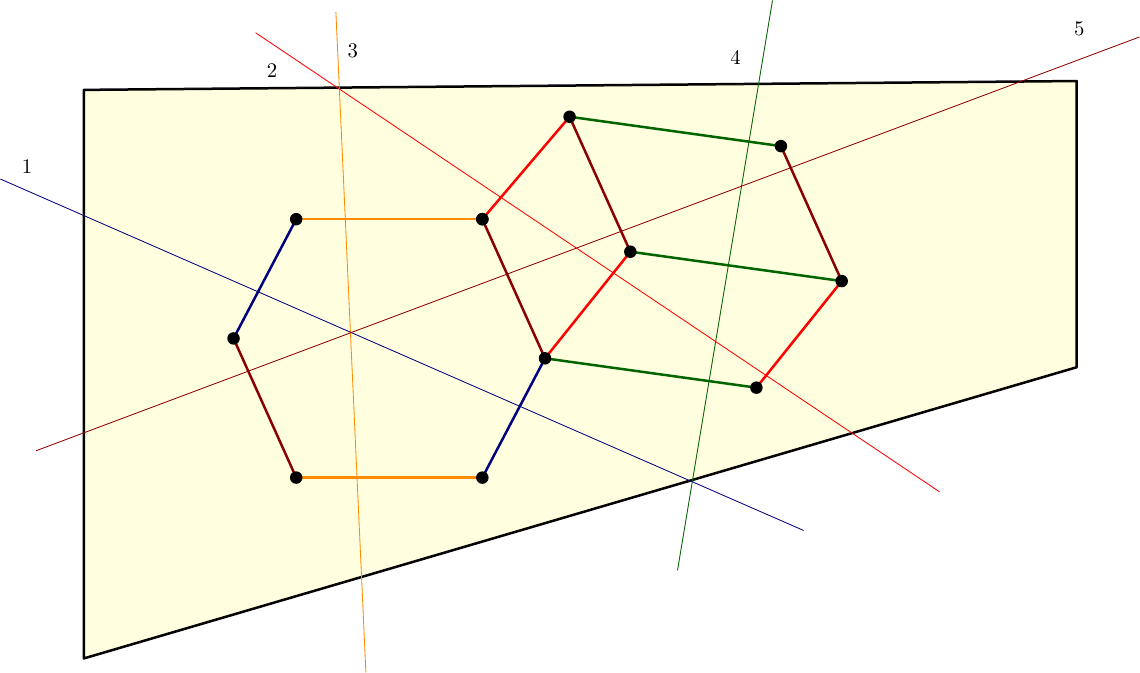}
\caption{The tope graph of the COM from Figure~\ref{fig:COMexample}. The color classes of edges giving the embedding into the cube correspond to the hyperplanes of the arrangement in $\mathbb{R}^2$.}
\label{fig:COMtopegraphexample}
\end{figure}

The \emph{tope graph} of a system of sign vectors is the graph induced by its topes, i.e., covectors without zero entries, in the hypercube $\{+,-\}^{\mathcal{E}}$. In case of COMs the tope graph (without labeling of the vertices) determines the COM up to isomorphism, see~\cite{Ban-15}. Moreover, the tope graph of a COM is a partial cube, see~\cite{Ban-15}. One way of thinking of being a partial cube is that the edges receive colors corresponding to the dimensions of the hypercube, see Figure~\ref{fig:COMtopegraphexample}.

Of particular importance to us are two types of metric subgraphs. An \emph{antipodal} subgraph $H$ of $G$ has the property, that for each vertex $v$ in $H$ there is an antipode $-_{H}v$, such that the $H$ is smallest convex subgraph of $G$ containing $v$ and $-_{H}v$. The antipodal subgraphs of the graph in Figure~\ref{fig:COMtopegraphexample} are exactly the vertices, edges, and bounded faces. The second property of a subgraph $H$ is the one of being \emph{gated}. This means, that for every vertex $v\in G$, there is a gate $v'$ in $H$, such that for every $v''\in H$ there is a shortest path from $v$ through $v'$ to $v''$. In partial cubes this amounts to the fact that there is path from $v$ to $H$ that does not use any color that is present on the edges of $H$. We say that a graph is \emph{antipodally gated} if all of its antipodal subgraphs are gated. The graph in Figure~\ref{fig:COMtopegraphexample} is antipodally gated, but also for instance the subgraph induced by the vertices incident to red or green edges is gated. A non-gated subgraph is given by the path $P$ of length two induced by the three left-most vertices. A vertex that has no gate in $P$ is the degree four vertex $v$, since all paths from $v$ to $P$ use a color present in $P$.

Exploring correspondences between axiomatical behavior of sign-vectors and metric subgraphs of partial cubes, in Theorem~\ref{thm:dictionaryOMC} we show that tope graphs of COMs are antipodally gated. This is the first part of the proof of Theorem~\ref{thm:main}, i.e., $\topegraphs\subseteq \AG$.

Clearly, there are partial cubes that do not satisfy Theorem~\ref{thm:dictionaryOMC}, i.e., they contain a non-gated antipodal subgraph. Figure~\ref{fig:pcandminorsxmpl} shows a partial cube in which the bottom $C_6$ is an antipodal subgraph, that is not gated, because the red vertex does not have gate.

\begin{figure}[htb]
\centering
\includegraphics[width=.7\textwidth]{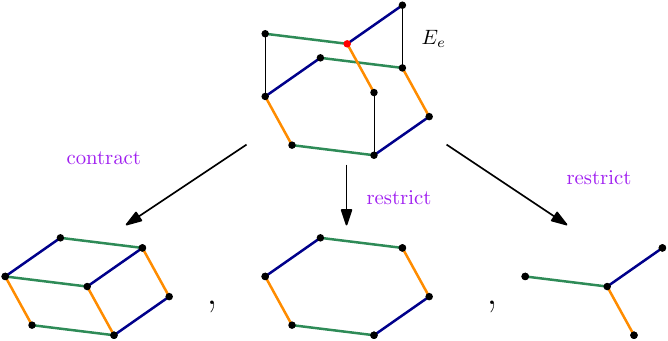}
\caption{A non antipodally gated partial cube and the partial cube minors obtained by contracting or restricting with respect to the vertical edge class $E_e$.}
\label{fig:pcandminorsxmpl}
\end{figure}

%

Instead of studying all the partial cubes, that have a non-gated antipodal subgraph, we use the notion of partial cube minors, in order to classify only minimal such partial cubes with respect to this operation. A \emph{partial cube minor} is either a contraction of a color class or the restriction to one its sides. Hence, it is a specialization of the standard graph minor notion. In Figure~\ref{fig:pcandminorsxmpl} we illustrate the partial cube minor on an example.

The next and second step of our proof is providing a set $\Q$ of partial cubes that are minor-minimal with respect to having a non-gated antipodal subgraph. The graph of Figure~\ref{fig:pcandminorsxmpl} is the smallest element of $\Q$, more of these graphs are depicted in Figure~\ref{fig:obstr}. It is easy to check that the minors in Figure~\ref{fig:pcandminorsxmpl} are antipodally gated, i.e., the graph on top is minimally non antipodally gated. The class of tope graphs of COMs is closed under pc-minors. This is illustrated by our realizable example in Figure~\ref{fig:COMminorexample}. By definition the class $\mathcal{F}(\Q)$ of partial cubes excludes the minors from $\Q$ as well. In Theorem~\ref{thm:topegraphs} we use this to show that, if $G$ is not the tope graph of a COM, then it must have a partial cube minor from $\Q$. This concludes the second part of our proof since it means, that if $G$ excludes $\Q$, then it is the tope graph of a COM, i.e., $\mathcal{F}(\Q)\subseteq \topegraphs$.

\begin{figure}[htb]
\centering
\includegraphics[width=\textwidth]{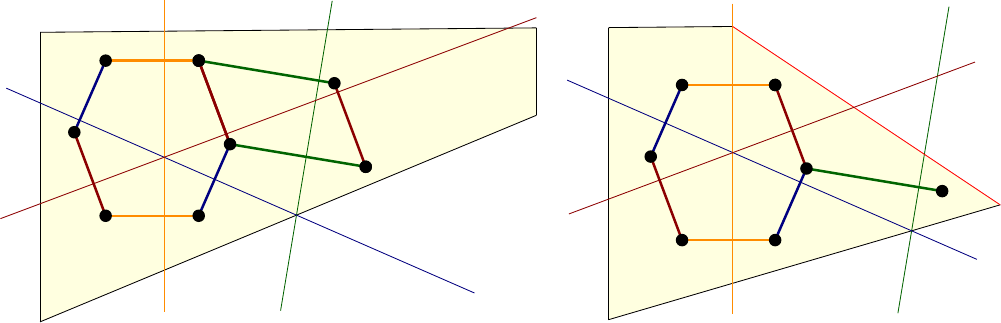}
\caption{Two partial cube minors obtained from the COM of Figure~\ref{fig:COMtopegraphexample} by contracting and restricting with respect to the red color class.}
\label{fig:COMminorexample}
\end{figure}

The last and most technical part of our proof is to show that the class of antipodally gated partial cube does not contain a member of $\Q$ as partial cube minor. The class $\mathcal{F}(\Q)$ is closed under partial cube minors and graphs in $\Q$ are minor-minimal having non-gated antipodal subgraphs. We will thus show that the class of antipodally gated partial cubes is closed under partial cube minors (Theorem~\ref{thm:AGminors}). This implies, that if $G$ has a partial cube minor from $\Q$, then it cannot be antipodally gated, i.e., $\AG\subseteq \mathcal{F}(\Q)$. This concludes the last part of the circular proof of our main theorem, which we will sketch quickly again.

\begin{thm}\label{thm:main}
 For a graph $G$ the following conditions are equivalent:
 \begin{itemize}
  \item[(i)] $G$ is the tope graph of a COM, i.e., $G\in\topegraphs$,
  \item[(ii)] $G$ is an antipodally gated partial cube, i.e., $G\in\AG$,
  \item[(iii)] $G$ is a partial cube with no partial cube minor from $\Q$, i.e., $G\in\mathcal{F}(\Q)$.
 \end{itemize}
\end{thm}
\begin{proof}
 The implication (i)$\Rightarrow$(ii) is Theorem~\ref{thm:dictionaryOMC}. The implication (iii)$\Rightarrow$(i) is Theorem~\ref{thm:topegraphs}. Finally, (ii)$\Rightarrow$(iii) follows from the fact that all the graphs in $\Q$ have a non-gated antipodal subgraph asserted in Lemma~\ref{lem:minfamily} and that the class $\AG$ is pc-minor closed, see Theorem~\ref{thm:AGminors}. 
\end{proof}

In order to get an idea of the further implications of this theorem without going into the technical details we propose to jump directly to Section~\ref{sec:corollaries}. The first corollary will establish a generalization of a theorem of Handa and characterize tope graphs of COMs by the fact that all iterated zone graphs are partial cubes. Afterwards resulting characterizations of tope graphs of OMs, LOPs, and AOMs (of bounded rank) are given. Moreover, the polynomial time recognition is shown and answering a conjecture of~\cite{Che-16} it is shown that Pasch graphs are tope graphs of COMs. 


\bigskip

In the following we survey in more detail the \textbf{Structure of the paper:} 

\bigskip

The next two sections (\ref{sec:pcminors} and~\ref{sec:signvectors}) are preliminaries dedicated on the one hand to partial cubes and on the other hand to systems of sign vectors. They can be skipped and looked back at when necessary. Apart from this purpose a couple of results of independent interest are given in Section~\ref{sec:pcminors}.

In Section~\ref{sec:pcminors} we introduce partial cubes with some more care, as well as metric subgraphs such as convex, gated, antipodal, and affine subgraphs and we discuss their behavior with respect to pc-minors and expansions. This section is quite technical and heavy in definitions since it introduces all the necessary background and auxiliary lemmas, that are needed for the rest of the paper. We discuss zone graphs of partial cubes, which play a role in part of our proof. We devise a polynomial time algorithm for checking if a given partial cube has another one as a pc-minor (Proposition~\ref{prop:polyminor}). We give an expansion procedure of how to construct all antipodal partial cubes from a single vertex (Lemma~\ref{lem:antipodalexpansion}) and provide an intrinsic characterization of affine partial cubes (Proposition~\ref{prop:affinepc}).

Section~\ref{sec:signvectors} is dedicated to the introduction of the systems of sign-vectors relevant to this paper, i.e.~COMs, OMs, AOMs, and LOPs, and their behavior under the usual minor-relations. Also, here quite some amount of definitions is introduced, that however coincides with the standard such as given in~\cite{Ban-15}.

In Section~\ref{sec:alltogethernow} we bring the content of the first two sections together and explain how systems of sign-vectors lead to partial cubes and vice versa. We show how metric properties of subgraphs and pc-minors correspond to axiomatic properties and minor relations of systems of sign-vectors. In particular we prove that tope graphs of COMs are antipodally gated (Theorem~\ref{thm:dictionaryOMC}), and characterize tope graphs of OMs, AOMs, and LOPs as special tope graphs of COMs. Theorem~\ref{thm:dictionaryOMC} gives the first implication for our characterization theorem (Theorem~\ref{thm:main}).

In Section~\ref{sec:forbidden} we introduce the (infinite) set of excluded pc-minors of tope graphs of COMs and provide some of its crucial properties, that will be used throughout the proofs in the following sections. In particular, we show that every member of the class has an antipodal subgraph that is not gated (Lemma~\ref{lem:minfamily}).
We conclude Section~\ref{sec:forbidden} with the the proof that partial cubes excluding all pc-minors from the class are tope graphs of COMs (Theorem~\ref{thm:topegraphs}). In particular, Theorem~\ref{thm:topegraphs} gives the second implication of Theorem~\ref{thm:main}. 

Finally, in Section~\ref{sec:antipodalminorclosed} we show that the class of antipodally gated partial cubes under pc-minors (Theorem~\ref{thm:AGminors}). Since the members of our set $\Q$ of excluded pc-minors have non-gated antipodal subgraphs, this yields the third and last implication of Theorem~\ref{thm:main}.

Section~\ref{sec:corollaries} is dedicated to the corollaries of our theorem, that are announced above. In, particular we prove the generalization of Handa's Theorem (Corollary~\ref{cor:Handa}) and prove a conjecture of~\cite{Che-16} (Corollary~\ref{cor:S4}). We conclude the paper with several further questions in Section~\ref{sec:conclusions}.

\section{Pc-minors, expansions, zone graphs, and metric subgraphs}\label{sec:pcminors}
In the present section we will give a thorough introduction to the theory of partial cubes and its elements that are important to our characterization. This will contain many definitions and lemmas, that will be used later. Of central importance are partial cube minors, expansions, zone graphs, and their interactions with metric subgraphs, such as convex, antipodal, affine, and gated subgraphs.


%

%
%
%

Let us start by giving an alternative way of characterizing partial cubes. 
Any isometric embedding of a partial cube into a hypercube leads to the same partition of edges into so-called $\Theta$-classes, where two edges are equivalent, if they correspond to a change in the same coordinate of the hypercube. This can be shown using the Djokovi\'c-Winkler-relation $\Theta$ which is defined in the graph without reference to an embedding, see~\cite{Djo-73,Win-84}. We will describe next, how the relation $\Theta$ can be defined independently of an embedding.

A subgraph $G'$ of $G$ is \emph{convex} if for all pairs of vertices in $G'$ all their shortest paths in $G$ stay in $G'$. 
For an edge $a=uv$ of $G$, define the sets $W(u,v)=\{ x\in V: d(x,u)<d(x,v)\}$. 
By a theorem of Djokovi\'c~\cite{Djo-73}, a graph $G$ is a partial cube if and only if $G$ is bipartite and for any edge $a=uv$
the sets $W(u,v)$ and $W(v,u)$ are convex. In this case, setting $a\Theta a'$ for $a=uv$ and $a'=u'v'$ if $u'\in W(u,v)$ and $v'\in W(v,u)$ yields $\Theta$.

We index the set of equivalence classes of $\Theta$ by a set ${\mathcal E}$. For $f\in\mathcal{E}$ we denote the equivalence class by $E_f$.For an arbitrary (oriented) edge $uv\in E_f$, let $E^-_f:=W(u,v)$ and $E^+_f:=W(v,u)$ the pair of complementary
convex halfspaces of $G$.
Now, identifying any vertex $v$ of $G$ with $v\in Q_{\mathcal{E}}=\{+, -\}^{\mathcal{E}}$ which for any class of $\Theta$ associates the sign of the halfspace containing $v$ gives an isometric embedding of $G$ into $Q_{\mathcal{E}}$.

\subsection{Pc-minors, expansions, and zone graphs}
We will now introduce the notions of Pc-minors (i.e., contraction and restriction) and zone graphs, which are methods to obtain smaller partial cubes from bigger ones, as well as, expansions, which are inverses of contractions. An important observation in this section is a polynomial time algorithm for checking for a given partial cube minor (Proposition~\ref{prop:polyminor}).

\subsubsection{Restrictions}

Given $f\in \mathcal{E}$,  an \emph{(elementary) restriction} consists in taking one of the subgraphs $G[E^-_f]$  or $G[E^+_f]$ induced by the
complementary halfspaces $E^-_f$ and $E^+_f$, which we will denote by $\rho_{f^-}(G)$ and $\rho_{f^+}(G)$, respectively. These graphs are isometric subgraphs of the hypercube $Q_{\mathcal{E}\setminus \{ f\}}$.
Now applying two elementary restriction with respect to different coordinates $f,g$, independently of the order of $f$ and $g$,
we will obtain one of the four (possibly empty) subgraphs induced by $E^-_f\cap E^-_g,E^-_f\cap E^+_g,E^+_f\cap E^-_g,$ and $E^+_f\cap E^+_g$. Since the intersection of convex subsets is convex, each of these
four sets is convex in $G$ and consequently induces an isometric subgraph of the hypercube $Q_{\mathcal{E}\setminus\{ f,g\}}$. More generally, a \emph{restriction} is a subgraph of $G$ induced by the
intersection of a set of (non-complementary) halfspaces of $G$. 
See Figures~\ref{fig:pcandminorsxmpl},~\ref{fig:COMminorexample}, and~\ref{fig:convexsubs} for examples of restrictions.  We denote restrictions by $\rho_{X}(G)$, where $X\in\{+, -\}^{\mathcal{E}}$ is a signed set of halfspaces of $G$.
For subset $S$ of the vertices of $G$ and $f\in \mathcal{E}$, we denote $\rho_{f^+}(S):=\rho_{f^+}(G)\cap S$ and $\rho_{f^-}(S):=\rho_{f^-}(G)\cap S$, respectively. We will say that $E_f$ \emph{crosses} a subset of vertices $S$ of $G$ if $\rho_{f^+}(S)\ne \emptyset$ and $\rho_{f^-}(S)\ne \emptyset$. 

The smallest convex subgraph of $G$ containing $V'$ is called the \emph{convex hull} of $V'$ and denoted by $\conv(V')$.
The following is well-known, also see Figure~\ref{fig:convexsubs}: 

\begin{lem}[\cite{Alb-16,Ban-89,Che-86}] \label{lem:restriction}
 The set of restrictions of a partial cube $G$ coincides with its set of convex subgraphs. Indeed, for any subset of vertices $V'$ we have that $\conv(V')$ is the intersection of all halfspaces containing $V'$. In particular, the class of partial cubes is closed under taking restrictions.
\end{lem}

\begin{figure}[htb]
\centering
\includegraphics[width=.7\textwidth]{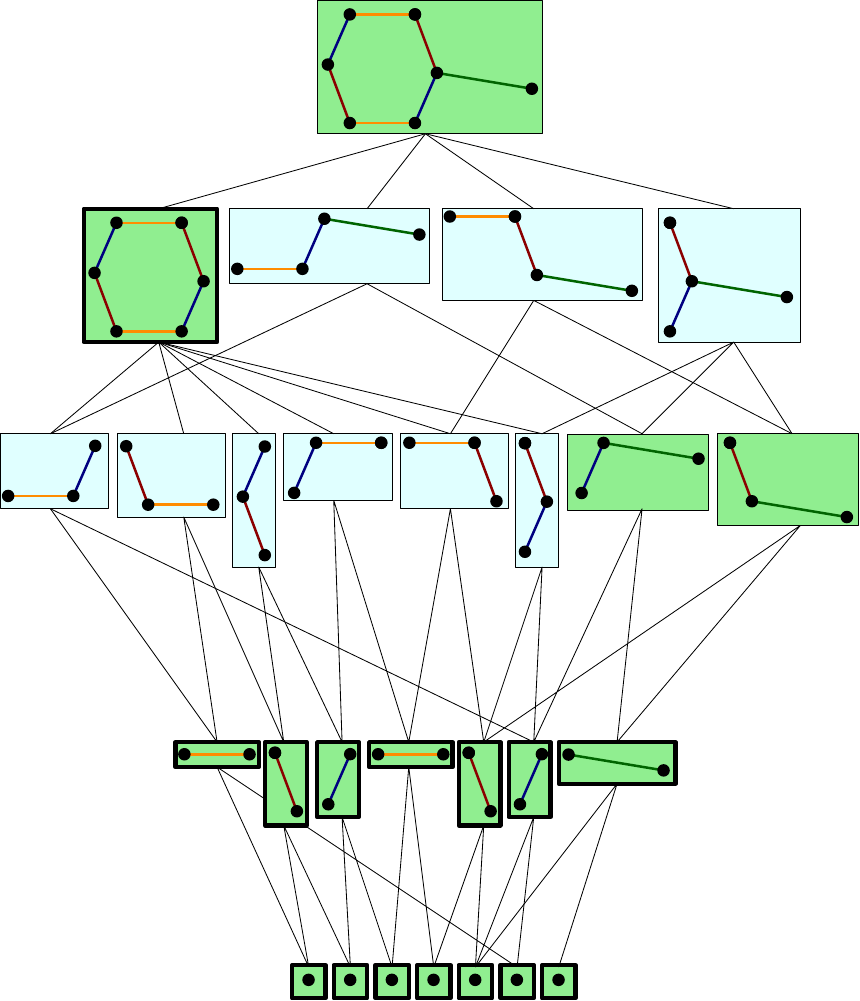}
\caption{The convex subgraphs of a partial cube ordered by inclusion. Green background means gated, thick outline means antipodal.}
\label{fig:convexsubs}
\end{figure}

%
%

\subsubsection{Contractions}

For $f\in \mathcal{E}$, we say that the graph  $G/E_f$ obtained from
$G$ by contracting the edges of the equivalence class $E_f$ is an \emph{(elementary) contraction} of $G$. For a vertex $v$ of $G$, we will denote by $\pi_f(v)$ the image of $v$ under the contraction in $G/E_f$, i.e.~if $uv$ is an
edge of $E_f$, then $\pi_f(u)=\pi_f(v)$, otherwise $\pi_f(u)\ne \pi_f(v)$. We will apply $\pi_f$ to subsets $S\subseteq V$, by setting $\pi_f(S):=\{\pi_f(v): v\in S\}$. In particular we denote the \emph{contraction} of $G$ by $\pi_f(G)$. See Figures~\ref{fig:pcandminorsxmpl},~\ref{fig:COMminorexample}, and~\ref{fig:contractions} for examples of contractions.

It is well-known and in particular follows from the proof of the first part of~\cite[Theorem 3]{Che-88} that $\pi_f(G)$ is an isometric subgraph of $Q_{\mathcal{E}\setminus \{ f\}}$. Since  edge contractions in graphs commute, i.e.~the resulting graph does not depend on the order in which a set of edges is contracted, we have:

\begin{lem}\label{lem:commut_contraction}
The class of partial cubes is closed under contractions. Moreover, contractions commute in partial cubes, i.e.~if $f,g\in \mathcal{E}$ and $f\ne g$, then $\pi_g(\pi_f(G))=\pi_f(\pi_g(G))$. 
\end{lem}
%

Consequently, for a set $A\subseteq \mathcal{E}$, we denote by $\pi_A(G)$ the isometric subgraph of $Q(\mathcal{E}\setminus A)$ obtained from $G$ by contracting the classes $A\subseteq \mathcal{E}$ in $G$. See Figure~\ref{fig:contractions} for examples.

\begin{figure}[htb]
\centering
\includegraphics[width=.7\textwidth]{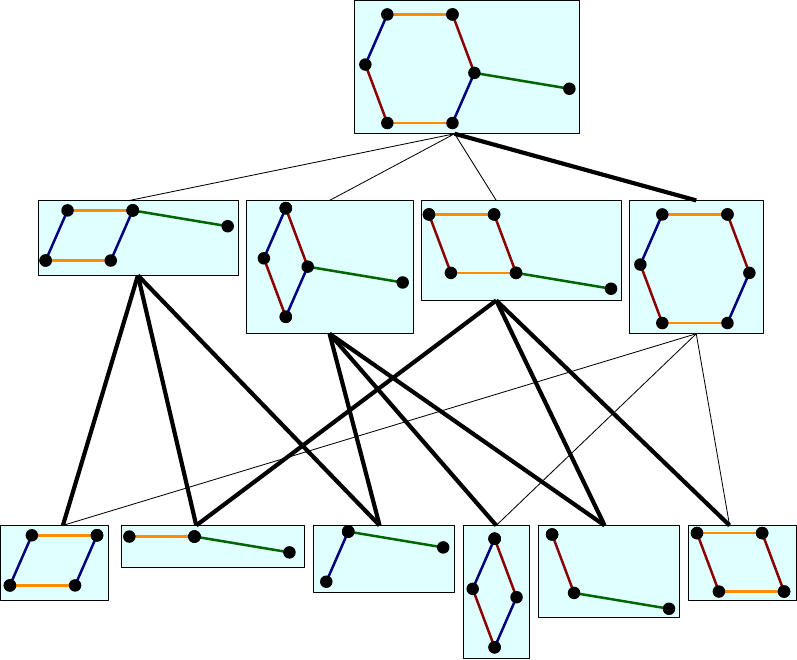}
\caption{Some contractions of a partial cube. Thick edges mean peripherality.}
\label{fig:contractions}
\end{figure}

The following can easily be derived from the definitions, see e.g.~\cite{Che-16}:

\begin{lem}\label{lem:commut_rest_contaction}
Contractions and restrictions commute in partial cubes, i.e.~if $f,g\in \mathcal{E}$ and $f\ne g$, then $\rho_{g^+}(\pi_f(G))=\pi_f(\rho_{g^+}(G))$.
\end{lem}
%

The previous lemmas show that any set of restrictions and any set of contractions of a partial cube $G$ provide the same result, independently of the order in which we perform the restrictions and contractions. The resulting graph $G'$ is also a partial cube, and $G'$ is called a \emph{pc-minor} of $G$. In this paper we will study classes of partial cubes that are closed under taking pc-minors. Clearly, any such class has a (possibly infinite) set ${X}$ of minimal excluded pc-minors. We denote by $\mathcal{F}(X)$ the pc-minor closed class of partial cubes excluding $X$.

\begin{prop}\label{prop:polyminor}
 Let $X$ be a finite set of partial cubes. It is decidable in polynomial time if a partial cube $G$ is in $\mathcal{F}(X)$. 
\end{prop}
\begin{proof}
Let $G',G$ be partial cubes. Denote by $n'$ and $n$ the number of vertices of $G'$ and $G$, respectively, and with $k'$ and $k$ the number of $\Theta$-classes in $G'$ and $G$. We will show that testing if $G'$ is a pc-minor of $G$ can be done in polynomial time with respect to $n$. This clearly implies the result. 

For every subset $V'$ of at most $n'$ vertices of $G$ do the following: First compute $\conv(V')$ and count the number of $\Theta$-classes of $G$ crossing it, say it equals to $k''$. Then $k''\leq k$, and if $k'' < k'$ discard the subgraph. On the other hand, if $k'' \geq k'$, then for every subset $S$ of size $k''-k'$ of the $\Theta$-classes crossing  $\conv(V')$, contract  in $\conv(V')$ all the $\Theta$-classes  of $S$. Finally, check if the resulting graph is isomorphic to $G'$.
 
 Using Lemma~\ref{lem:commut_rest_contaction}, we know that $G'$ is a pc-minor if and only if it can be obtained by first restricting and then contracting, and by Lemma~\ref{lem:restriction}, taking restrictions coincides with taking convex hulls. We need to prove that one can take a convex hull of exactly $n'$ vertices. For that assume that $G'$ can be obtained from $G$ by first restricting to $G''$ and then contracting. For each vertex $v'$ of $G'$ pick a vertex $v''$ in $G''$ such that $v''$ maps to $v'$ under contraction. Then the set of all such $v''$ is also a subset in $G$, call it $S$. We claim that taking a convex hull of $S$ and then contracting gives $G'$. In fact, the obtained graph must be a subgraph of $G'$ since the convex hull of $S$ is a subset of $G''$. On the other hand, every vertex of $G'$ is obtained in this way since $S$ includes a representative of pre-image of every vertex of $G'$. This gives the correctness of the algorithm.
 
 For the running time, we have a loop of length $\mathcal{O}(n^{n'})$. In each execution we compute $\conv(V')$ which via Lemma~\ref{lem:restriction} can be easily done by intersecting all the halfspaces containing $V'$. Then we have $\mathcal{O}( {{k''}\choose{k''-k'}}) = \mathcal{O}({ {k''}\choose {k'}}) \leq \mathcal{O}({{ n }\choose{ k'}}) = \mathcal{O}(n^{k'})$ choices for the contractions of the $\Theta$-classes, each of which  can clearly be done in polynomial time, too. Note that $k'<n'$. Finally, we check if the obtained graph is isomorphic to $ G'$, which only depends on $n'$.
\end{proof}



\subsubsection{Expansions}

Later on we will also consider the inverse operation of contraction: a partial cube $G$ is an \emph{expansion} of a partial cube $G'$ if $G'=\pi_f(G)$ for some $\Theta$-class $f$ of $G$. Indeed expansions can be detected within the smaller graph. Let $G'$ be a partial cube containing two isometric subgraphs $G'_1$ and $G'_2$ such that $G'=G'_1\cup G'_2$, there are no edges from $V(G'_1\setminus G'_2)$ to $V(G'_2\setminus G'_1)$, and denote by $G'_0:=G[V(G'_1)\cap V(G'_2)]$ the subgraph induced by the vertices that are in both $G'_1$ and $G'_2$.  A graph $G$ is an expansion
of  $G'$ with respect to $G_0$ 
if $G$ is obtained from $G'$ by replacing each vertex $v$ of $G'_1$ by a vertex $v_1$ and each vertex $v$ of  $G'_2$ by a vertex $v_2$ such that $u_i$ and $v_i$, $i=1,2$ are adjacent in $G$ if and only if $u$ and $v$ are adjacent vertices of $G'_i$, and $v_1v_2$ is an edge of $G$ if and only if $v$ is a vertex of $G'_0$. The following is well-known:

\begin{lem}[\cite{Che-86,Che-88}]\label{lem:allfromone}
A graph $G$ is a partial cube if and only if $G$ can be obtained by a sequence of expansions from a single vertex.
\end{lem}

We will make use of the following lemma about the interplay of contractions and expansions:

\begin{lem}\label{lem:G1G2}
Assume that we have the following commutative diagram of contractions:

\begin{center}
\begin{tikzcd}
G \arrow{r}{\pi_{f_1}} \arrow{d}{\pi_{f_2}} & \pi_{f_1} (G) \arrow{d}{\pi_{f_2}} \\%
\pi_{f_2} (G) \arrow{r}{\pi_{f_1}}& \pi_{f_1} (\pi_{f_2} (G))
\end{tikzcd}
\end{center}
If $G$ is expanded from $\pi_{f_1} (G)$ along sets $G_1,G_2\subseteq \pi_{f_1} (G)$, then $\pi_{f_2} (G)$ is expanded from $\pi_{f_1} (\pi_{f_2} (G))$ along sets $\pi_{f_2} (G_1)$ and $\pi_{f_2} (G_2)$.

\end{lem}

\begin{proof}
Let $\pi_{f_2} (G)$ be expanded from $\pi_{f_1} (\pi_{f_2} (G))$ along sets $H_1,H_2$. Consider $v\in \pi_{f_1} (\pi_{f_2} (G))$. Vertex $v$ is in $H_1\cap H_2$ if and only if its preimage in $\pi_{f_2} (G)$ is an edge $a\in E_{f_1}$. This is equivalent to $\pi_{f_1}^{-1}(a)$ being intersected by $E^+_{f_1}$ and $E^-_{f_1}$ in $G$. But this means, that the image of $\pi_{f_2}^{-1}(a)$ in $\pi_{f_1} (G)$, say $I:=\pi_{f_1}(\pi_{f_2}^{-1}(a))$, has at least one vertex in $G_1\cap G_2$. The image $I$ is contracted to $v$ by $\pi_{f_2}$, thus $I$ is an edge or a vertex. Since every edge of $\pi_{f_1} (G)$ must have both its endpoints in $G_1$ or both its endpoints in $G_2$, we deduce that $I$ has a vertex in $G_1\cap G_2$ if and only if $v$ in $\pi_{f_2}(G_1)\cap \pi_{f_2}( G_2 )$.  This proves that $H_1\cap H_2 = \pi_{f_2}(G_1)\cap \pi_{f_2}(G_2)$.

Removing  $H_1\cap H_2= \pi_{f_2}(G_1)\cap \pi_{f_2}(G_2)$ from $\pi_{f_1} (\pi_{f_2} (G))$ cuts it into two connected components, one a subset of $H_1$, one a subset of  $H_2$. On the other hand, removing $G_1\cap G_2$ from $\pi_{f_1} (G)$ also cuts it into two connected components, one in $G_1$ and one in $G_2$. Since $\pi_{f_2}$ maps connected subgraphs to connected subgraphs, we see that $H_1=\pi_{f_2}(G_1)$ and $H_2=\pi_{f_2}(G_2)$, or the other way around.
\end{proof}

Let $G$ be a partial cube and $f\in\mathcal{E}$ indexing one of its $\Theta$-classes $E_f$. Assume that a halfspace $E_f^+$ (or $E_f^-$) is such that all its vertices are incident with edges from  $E_f$. Then we call  $E_f^+$ (or $E_f^-$) \emph{peripheral}. In such a case we will also call $E_f$ a peripheral $\Theta$-class, and call $G$ a \emph{peripheral expansion}  of $\pi_f(G)$. Note that an expansion along sets $G_1,G_2$ is peripheral if and only if one of the sets $G_1,G_2$ is the whole graph and the other one an isometric subgraph. An expansion is called \emph{full} if $G_1=G_2$. Note that in this case, the expanded graph is isomorphic to $G_1\square K_2$. See Figure~\ref{fig:contractions} for examples of peripheral and non-peripheral expansions.

\subsubsection{Zone graphs}

For a partial cube $G$ and $f\in\mathcal{E}$ the \emph{zone graph} of $G$ with respect to $f$ is the graph $\zeta_f(G)$ whose vertices correspond to the edges of $E_f$ and two vertices are connected by an edge if the corresponding edges of $E_f$ lie in a convex cycle of $G$, see~\cite{Kla-12}. Here, a \emph{convex cycle} is just a convex subgraph that is a cycle. In particular, $\zeta_f$ can be seen as a mapping from edges of $G$ that are not in $E_f$ but lie on a convex cycle crossed by $E_f$ to the edges of $\zeta_f(G)$. If $\zeta_f(G)$ is a partial cube, then we say that $\zeta_f(G)$ is \emph{well-embedded} if for two edges $a,b$ of $\zeta_f(G)$ we have $a\Theta b$ if and only if the sets of $\Theta$-classes crossing $\zeta_f^{-1}(a)$ and $\zeta_f^{-1}(b)$ coincide and otherwise they are disjoint. As an example, note that all zone graphs of the graph in Figure~\ref{fig:COMtopegraphexample} are well-embedded paths, while all zone graphs of the graph on top in Figure~\ref{fig:pcandminorsxmpl} are triangles. For yet another example, see Figure~\ref{fig:zonegraphs}. A consequence of Corollary~\ref{cor:Handa} will be that out of these three examples only the first one is the tope graph of a COM.

\begin{figure}[htb]
\centering
\includegraphics[width=.6\textwidth]{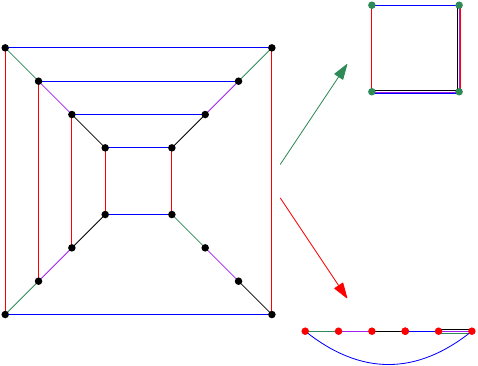}
\caption{A partial cube all of whose zone graphs are either 4-cycles or 6-cycles, but none of them are well-embedded.}
\label{fig:zonegraphs}
\end{figure}

For discussing zone graphs in partial cubes the following will be useful.
Let $v_1u_1, v_2u_2 \in E_e$ be edges in a partial cube $G$ with $v_1,v_2\in E_e^+$. Let $C_1,\ldots, C_n$, $n\geq 1$, be a sequence of convex cycles such that $v_1u_1$ lies only on $C_1$, $v_2u_2$ lies only on $C_n$, and each pair $C_i$ and $C_{i+1}$, for $i\in \{1,\ldots,n-1\}$, intersects in exactly one edge and this edge is in $E_e$, all the other pairs do not intersect. If the shortest path from $v_1$ to $v_2$ on the union of $C_1,\ldots, C_n$ is a shortest $v_1, v_2$-path in $G$, then we call $C_1,\ldots, C_n$ a \emph{convex traverse} from $v_1u_1$ to $v_2u_2$. In~\cite{Mar-16} it was shown that for every pair of edges $v_1u_1, v_2u_2$ in relation $\Theta$ there exists a convex traverse connecting them.

\begin{lem}\label{lem:pchyperplane}
 Let $G$ be a partial cube and $f\in\mathcal{E}$. Then $\zeta_f(G)$ is a well-embedded partial cube if and only if for any two convex cycles $C, C'$ that are crossed by $E_f$ and some $E_g$ both $C$ and $C'$ are crossed by the same set of $\Theta$-classes.
\end{lem}
\begin{proof}
 The direction ``$\Rightarrow$'' follows immediately from the definition of well-embedded. 
 
 For ``$\Leftarrow$'' let $G$ satisfy the property that for any two convex cycles $C, C'$ that are crossed by $E_f$ and some $E_g$ both $C, C'$ are crossed by the same set of $\Theta$-classes. 
 
 Define an equivalence relation on the edges of $\zeta_f(G)$ by $a\sim b$ if and only if $\zeta_f^{-1}(a)$ and $\zeta_f^{-1}(b)$ are crossed by the same set of $\Theta$-classes. Let $a',b'\in E_f$ be two edges of $G$ corresponding to vertices of $\zeta_f(G)$. Then, there exists a convex traverse $T$ from $a'$ to $b'$, i.e., no two cycles in $T$ share $\Theta$-classes apart from $f$. By the property on convex cycles in $G$ all such paths from $a'$ to $b'$ in $\zeta_f(G)$ are crossed by the same set of equivalence classes and each exactly once. Furthermore, if there was a path in $\zeta_f(G)$ not corresponding to a traverse, its cycles would repeat $\Theta$-classes of $G$, thus cross several times equivalence classes of $\zeta_f(G)$. Thus, every equivalence class of $\sim$ cuts $\zeta_f(G)$ into two convex subgraphs. We have that $\zeta_f(G)$ is a partial cube and the embedding we defined shows that it is well-embedded.
\end{proof} 
 
 A well-embedded zone graph $\zeta_f(G)$ thus induces an equivalence relation on the $\Theta$-classes of $G$ except $f$, that are involved in convex cycles crossed by $E_f$. We denote by $\overline{e}$ the class of $\Theta$-classes containing $E_e$. Note that $\overline{e}$ corresponds to a $\Theta$-class of $\zeta_f(G)$ and vice versa.
 
 The following will be useful:
\begin{lem}\label{lem:parallel}
Let $\zeta_f(G)$ be a well-embedded partial cube and $g,h$ two equivalent $\Theta$-classes of $G$ and $C$ a convex cycle  crossed by $E_f,E_g,E_h$. If $a\in E_f$ is an edge such that $a = C\cap E_f\cap E_g^+ = C\cap E_f\cap E_h^+$, then each edge of $E_f$ is either in $E_g^+\cap E_h^+$ or in $E_g^-\cap E_h^-$.
\end{lem}
\begin{proof}
Suppose otherwise, that there is an edge $b$ in $E_f\cap E_g^+\cap E_h^-$. The interval from $a$ to $b$ is crossed by $E_h$ but not by $E_g$. Let $T$ be a convex traverse from $a$ to $b$. Then there exists a convex cycle on $T$ crossed by $E_f$ and $E_h$ but not by $E_g$ contradicting Lemma \ref{lem:pchyperplane}. 
\end{proof}

 Lemma~\ref{lem:parallel} justifies that if $\zeta_f(G)$ is a well-embedded partial cube and  $g\in\mathcal{E}\setminus\{f\}$, then we can orient $\overline{g}$ in $\zeta_f(G)$ such that $\rho_{\overline{g}^+}(\zeta_f(G))=\zeta_f(\rho_{\overline{g}^+}(G))$.

\begin{lem}\label{lem:hyperplanescommute}
 Let $G$ be a partial cube, $f\in\mathcal{E}$ such that $\zeta_f(G)$ is a well-embedded partial cube, $A\subseteq\mathcal{E}\setminus\{f\}$ and $X\in\{+, -\}^A$. We have $\pi_{\{\overline{e}\, \mid \, e\subseteq A\}}(\zeta_f(G))=\zeta_f(\pi_A(G))$ and $\rho_{\{\overline{e}^{X_e}\mid e\in A\}}(\zeta_f(G))=\zeta_f(\rho_X(G))$. 
\end{lem}
\begin{proof}
For the contractions, clearly any contraction in $\zeta_f(G)$ corresponds to contracting the corresponding equivalence classes in $G$. Conversely, if some $\Theta$-classes $A$ are contracted in $G$, this affects only the classes of $\zeta_f(G)$ such that all the corresponding edges in $G$ are contained in $A$.

Taking a restriction in $\zeta_f(G)$ can be modeled by restricting to the respective sides of all the elements of the corresponding class of $\Theta$-classes of $G$.

By Lemma~\ref{lem:parallel}, if a set of restrictions in $G$ leads to a non-empty zone graph, there is an orientation for all the elements of the classes of $\Theta$-classes containing them leading to the same result.
\end{proof}

\subsection{Pc-minors and expansions versus metric subgraphs}
In this section we present  conditions under which contractions, restrictions, and expansions preserve metric properties of subgraphs, such as convexity, gatedness, antipodality, and affinity. An important result of the section is an intrinsic characterization of affine partial cubes (Proposition~\ref{prop:affinepc}).

\subsubsection{Convex subgraphs}

Let $G=(V,E)$ be an isometric subgraph of the hypercube $Q_{\mathcal{E}}$ and let $S$ be a subset of vertices of $G$. Let $f$ be any coordinate of $\mathcal{E}$. We will say that $E_f$ is \emph{disjoint} from $S$ if it does not cross $S$ and has no vertices in $S$. Note that a non-crossing class $E_f$ can have vertices in $S$, e.g., $S\cap E^+_f\ne\emptyset$. Thus, disjointness is stronger than to be non-crossing.
The following three lemmas describe the behavior of convex subgraphs under contractions, restrictions, and expansions. Their (short) proofs can be found in~\cite{Che-16}.

\begin{lem}\label{contraction_convex}
If $H$ is a convex subgraph of $G$ and $f\in\mathcal{E}$, then $\rho_{f^+}(H)$ is a convex subgraph of $\rho_{f^+}(G)$. If $E_f$ crosses $H$ or is disjoint from $H$, then also $\pi_f(H)$ is a convex subgraph of $\pi_f(G)$.
\end{lem}

\begin{lem}\label{convex_hull} If $S$ is a subset of vertices of $G$ and $f\in \mathcal{E}$, then $\pi_f(\conv(S))\subseteq \conv(\pi_f(S))$. If $E_f$ crosses $S$, then $\pi_f(\conv(S))= \conv(\pi_f(S))$.
\end{lem}
%

\begin{lem}\label{convex_expansion} If $H'$ is a convex subgraph of $G'$ and $G$ is obtained from $G'$ by an isometric expansion, then the expansion of $H$ of $H'$ is a convex subgraph of $G$.
\end{lem}
%

\subsubsection{Antipodal subgraphs}

Let $H$ be a subgraph of $G$. If for a vertex $x\in H$ there is a vertex $-_Hx\in H$ such that $\conv(x,-_Hx)=H$ we say that $-_Hx$ is the \emph{antipode} of $x$ with respect to $H$ and we omit the subscript $H$ if this causes no confusion. Intervals in a partial cube are convex since intervals in hypercubes equal (convex) subhypercubes, therefore $\conv(x,-_Hx)$ consists of all the vertices on the shortest paths connecting $x$ and $-_Hx$. Then it is easy to see, that if a vertex has an antipode, it is unique. We call a subgraph $H$ of a partial cube $G=(V,E)$ \emph{antipodal} if every vertex $x$ of $H$ has an antipode with respect to $H$. Note that antipodal graphs are sometimes defined in a different but equivalent way and then are called symmetric-even, see~\cite{Ber-88}. By definition, antipodal subgraphs are convex. See Figure~\ref{fig:convexsubs} for examples of antipodal and non-antipodal subgraphs. Their behavior with respect to pc-minors has been described in~\cite{Che-16} in the following way:

\begin{lem}\label{lem:antipodal_minor}
 Let $H$ be an antipodal subgraph of $G$ and $f\in\mathcal{E}$. If $E_f$ is disjoint from $H$, then $\rho_{f^+}(H)$ is an antipodal subgraph of $\rho_{f^+}(G)$. If $E_f$ crosses $H$ or is disjoint from $H$, then $\pi_f(H)$ is an antipodal subgraph of $\pi_f(G)$. 
\end{lem}

In particular, Lemma~\ref{lem:antipodal_minor} implies that the class of antipodal partial cubes is closed under contractions. Next we will deduce a characterization of those expansions that generate all antipodal partial cubes from a single vertex, in the same way as Lemma~\ref{lem:allfromone} characterizes all partial cubes. Let $G$ be an antipodal partial cube and $G_1, G_2$ two subgraphs corresponding to an isometric expansion. We say that it is an \emph{antipodal expansion} if and only if $-G_1=G_2$, where $-G_1$ is defined as the set of antipodes of $G_1$.

\begin{lem}\label{lem:antipodalexpansion}
Let $G$ be a partial cube and $\pi_e(G)$ antipodal. Then $G$ is an antipodal expansion of $\pi_e(G)$ if and only if $G$ is antipodal. In particular, all antipodal partial cubes arise from a single vertex by a sequence of antipodal expansion.
\end{lem}

\begin{proof}
Say $\pi_e(G)$ is expanded to $G$ along sets $G_1,G_2$. Let $v\in G_1$ and $v'\in G$ a vertex with $\pi_e(v')=v$. 

If $G$ is antipodal, there exists a vertex $-v'$ whose distance to $v'$ is equal to the number of $\Theta$-classes of $G$. In particular, the shortest path must cross $E_e$, proving that $\pi_e(-v')\in G_2$. But $\pi_e(-v') = -v$ proving that $-v\in G_2$. 

Conversely, if $-G_1=G_2$ it is easy to see, that the antipode of $v'$ is in $\pi^{-1}_e(-v)$.
\end{proof}

A further useful property of antipodal subgraphs of partial cubes proved in~\cite{Che-16} is the following:
\begin{lem}\label{lem:antipodal_cycle}
Let $H$ be an antipodal subgraph of $G$ and $u,v\in H$, then $H$ contains an isometric cycle $C$ through $v,u-_Hv$ such that $\conv(C)=H$.
\end{lem}

%
%
%

\subsubsection{Affine subgraphs}

We call a partial cube \emph{affine} if it is a halfspace of an antipodal partial cube. All graphs except the one on the top and the $K_{1,3}$ in Figure~\ref{fig:convexsubs} are affine. We can give the following intrinsic characterization of affine partial cubes, that will play a crucial role in our characterization of tope graphs of AOMs, see Corollary~\ref{cor:AOM}.

\begin{prop}\label{prop:affinepc}
 A partial cube $G$ is affine if and only if for all $u,v$ vertices of $G$ there are $w,-w$ in $G$ such that $\conv(u,w)$ and $\conv(v,-w)$ are crossed by disjoint sets of $\Theta$-classes.
\end{prop}
\begin{proof}
 Let $G=E^+_f(\widetilde{G})$ be a halfspace of an antipodal partial cube $\widetilde{G}$. For $u,v\in G$ consider the antipode $-_{\widetilde{G}}v$ of $v$ in $E^-_f(\widetilde{G})$. By Lemma~\ref{lem:antipodal_cycle}, we can consider an isometric cycle $C$ through $v,u,-_{\widetilde{G}}v$ such that $\conv(C)=\widetilde{G}$.  The two vertices $w,z$ on $C\cap E^+_f(\widetilde{G})$ that are incident with edges from $E_f(\widetilde{G})$ are connected on $C$ by a shortest path crossing all the $\Theta$-classes of $G$, i.e.~$z=-_Gw$. By symmetry of $w,-_Gw$, we can assume that $v$ appears before $u$ on a shortest path from $w$ to $-_Gw$. Thus $w,-_Gw\in G$ are such that $\conv(u,w)$ and $\conv(v,-_Gw)$ are crossed by disjoint sets of $\Theta$-classes.
 
 Conversely, let $G$ be such that for all $u,v\in G$ there are $w,-w\in G$ such that $\conv(u,w)$ and $\conv(v,-w)$ are crossed by disjoint sets of $\Theta$-classes. We construct $\widetilde{G}$ by taking a copy $G'$ of $G$ and join $w$ with an edge to $(-w)'$ for each pair $w,-w\in G$. Associating all these new edges to a new coordinate of the hypercube we get an embedding into a hypercube of dimension one higher. First we show that $\widetilde{G}$ is a partial cube. Since $G$ and its copy on their own are partial cubes, suppose now that $u\in G$ and $v'\in G'$. In $G$ we can take $w,-_Gw\in G$ such that $\conv_G(u,w)$ and $\conv_G(v,-_Gw)$ are crossed by disjoint sets of $\Theta$-classes. Consider a shortest path from $u$ to $w$, then the edge to $(-w)'$, and finally a shortest path from $(-w)'$ to $v'$. Since none of the original $\Theta$-classes was crossed twice, this is a shortest path of the hypercube that $\widetilde{G}$ is embedded in.

 It remains to show that $\widetilde{G}$ is antipodal. For every vertex $v\in G$ there exists $w,-w\in G$ such that $\conv(v,w)$ and $\conv(v,-w)$ are crossed by disjoint sets of $\Theta$-classes. In fact, in this case $\conv(v,w)$ and $\conv(v,-w)$ together cross all $\Theta$-classes of $G$. 
Hence taking a shortest path from $v$ to $w$, then the edge to $(-w)'$ and from there a shortest path to $v'$ yields a path from $v$ to $v'$ crossing each $\Theta$-class of $\widetilde{G}$ exactly once. This implies that $v'$ is an antipode of $v$.
 \end{proof}

By Lemma~\ref{lem:commut_rest_contaction} a contraction of a halfspace is a halfspace and by Lemma~\ref{lem:antipodal_minor} antipodal partial cubes are closed under contraction, therefore we immediately get:
 
 \begin{lem}\label{lem:affine_minor}
  The class of affine partial cubes is closed under contraction.
 \end{lem}

%

\subsubsection{Gated subgraphs}

A subgraph $H$ of $G$, or just a set of vertices of $H$, is called \emph{gated} (in $G$) 
if for every vertex $x$ outside $H$ there exists a vertex $x'$ in $H$, the \emph{gate} of $x$, such that each vertex $y$ of $H$ is
connected with $x$ by a shortest path passing through the gate
$x'$. It is easy to see that if $x$ has a gate in $H$, then it is unique
and that gated subgraphs are convex. See~\cite{Dre-87} for several results on gated sets in metric spaces. See Figure~\ref{fig:convexsubs} for examples of gated and non-gated subgraphs.

In~\cite{Che-16} it was shown that gated subgraphs behave well with respect to pc-minors:
\begin{lem}\label{contraction_gated}
If $H$ is a gated subgraph of $G$, then $\rho_{f^+}(H)$ and $\pi_f(H)$ are  gated subgraphs of $\rho_{f^+}(G)$ and $\pi_f(G)$, respectively.
\end{lem}
In the next lemma we will see that expansions can turn gated graphs into non-gated graphs.

\begin{lem}\label{lem:nongatedexpansion}
Let $G$ be an expansion of $\pi_e(G)$ along sets $G_1,G_2$. Let $H$ be a gated subgraph of $\pi_e(G)$, $v$ a vertex of $\pi_e(G)$ and $v'$ the gate of $v$ in $H$. If $v\in G_1\cap G_2$, $v'\notin G_1\cap G_2$ and there exist $v''\in H$, $v''\in G_1\cap G_2$, then the expansion of $H$ in $G$ is not gated.
\end{lem}

\begin{proof}
Let $v,v',v'',H$ be as in the lemma and without loss of generality assume that $v'\in G_1\backslash G_2$. Let $E_e^+$ correspond to $G_1$ and $E_e^-$ to $G_2$ in $G$. Since vertex $v\in G_1\cap G_2$, it is expanded to an edge in $G$. Let $u$ be the vertex on this edge in $E_e^+$. Then every shortest path form $u$ to the expansion of $H$ must cross at least the same $\Theta$-classes as a shortest path from $v$ to $v'$. On the other hand, $v,v'\in G_1$, thus in $G_1$ there exists a shortest path from $v$ to $v'$. Then there exists a shortest path from $u$ to the expansion of $H$, first crossing $E_e$ and then all the $\Theta$-classes in the shortest path from $v$ to  $v'$. Note that the expansion of $v'$ is not the gate of $u$ since there is no shortest path from $u$ to the expansion of $v''$ in $E_e^+$ passing this vertex. Thus if $u$ has a gate to the expansion of $H$, it must be at distance $d(v,v')$ to $u$ and the shortest path connecting them must be crossed by exactly those $\Theta$-classes that cross shortest paths from $v$ to $v'$. Then this gate must be adjacent to the expansion of $v'$ and be in $E_e^+$. This is impossible, since $v'\in G_1\backslash G_2$.
\end{proof}

\section{Systems of sign-vectors}\label{sec:signvectors}
In the present sections we will introduce the standard definitions concerning systems of sign-vectors. In particular we will introduce the axiomatics for COMs, AOMs, OMs, and LOPs and operations such as reorientations and minors. Recall that Figure~\ref{fig:COMexample} depicts the example of a COM.

We follow the standard OM notation from~\cite{bjvestwhzi-93} and concerning COMs we stick to~\cite{Ban-15}. Let $\mathcal{E}$ be a non-empty finite (ground) set and let $\emptyset\neq\covectors\subseteq\{+, - ,0\}^\mathcal{E}$. The elements of $\covectors$ are referred to as \emph{covectors}. 

For $X \in \covectors$, and $e\in \mathcal{E}$ let $X_e$ be the value of $X$ at the coordinate $e$. The subset $\underline{X} = \{e\in \mathcal{E}: X_e\neq 0\}$
is called the \emph{support} of $X$ and  its complement  $X^0=\mathcal{E}\setminus \underline{X}=\{e\in \mathcal{E}: X_e=0\}$ the \emph{zero set} of $X$. 
For $X,Y\in \covectors$, we call $S(X,Y)=\{f\in \mathcal{E}: X_fY_f=-\}$
the \emph{separator} of $X$ and $Y$. The \emph{composition} of $X$ and $Y$ is the sign-vector $X\circ Y,$ where
$(X\circ Y)_e = X_e$  if $X_e\ne 0$  and  $(X\circ Y)_e=Y_e$  if $X_e=0$. For a subset $A\subseteq\mathcal{E}$ and $X\in\covectors$ the \emph{reorientation} of $X$ with respect to $A$ is the sign-vector defined by $$
(_AX)_e:=\left\{\begin{array}{ll}
-X_e & \textrm{ if } e \in A\\
X_e &\textrm{ otherwise.}
\end{array}\right.
$$ In particular $-X:=_{\mathcal{E}}X$. The \emph{reorientation} of $\covectors$ with respect to $A$ is defined as $_A\covectors:=\{_AX\mid X\in\covectors\}$. In particular, $-\covectors:=_{\mathcal{E}}\covectors$.

We continue with the formal definition of the main axioms relevant for COMs, AOMs, OMs, and LOPs. All of them are closed under reorientation.

\medskip\noindent
{\bf Composition:}
\begin{itemize}
\item[{\bf (C)}] $X\circ Y \in  \covectors$  for all $X,Y \in  \covectors$.
\end{itemize}

Since $\circ$ is associative, arbitrary
finite compositions can be written without bracketing $X_1\circ\ldots \circ X_k$ so
that (C) entails that they all belong to $\covectors$. Note that contrary to a convention sometimes made in
OMs we do not consider compositions over an empty index set, since this would
imply that the zero sign-vector belonged to $\covectors$. The same consideration applies for the following two strengthenings of (C). 

\medskip\noindent
{\bf Face symmetry:}
\begin{itemize}
\item[{\bf (FS)}] $X\circ -Y \in  \covectors$  for all $X,Y \in  \covectors$.
\end{itemize}


By (FS) we first get $X\circ -Y\in\covectors$ and then $X\circ Y= (X\circ -X)\circ Y= X\circ -(X\circ -Y) \in \covectors$ for all $X,Y\in \covectors$. Thus, (FS) implies (C).


\medskip\noindent
{\bf Ideal composition:}
\begin{itemize}
\item[{\textbf{(IC)} }] $X\circ Y \in  \covectors$  for all $X \in  \covectors$ and $Y\in\{+, -, 0\}^\mathcal{E}$.
\end{itemize}

Note that (IC) implies (C) and (FS). The following axiom is part of all the systems of sign-vectors discussed in the paper:

\medskip\noindent
{\bf Strong elimination:}
\begin{itemize}
\item[{\bf (SE)}]  for each pair $X,Y\in\covectors$ and for each $e\in  S(X,Y)$ there exists $Z \in  \covectors$ such that
$Z_e=0$  and  $Z_f=(X\circ Y)_f$  for all $f\in \mathcal{E}\setminus S(X,Y)$.
\end{itemize}


%



An axiom particular to OMs is:

\medskip\noindent
{\bf Zero vector:}
\begin{itemize}
\item[{\bf (Z)}]  the zero sign-vector ${\bf 0}$ belongs to $\covectors$.
\end{itemize}

We will not make proper use of the axiomatization of AOMs due to~\cite{Bau-16,Kar-92} apart from illustrating  that AOMs are a natural subclass of COMs, see Definition~\ref{def:main}. 
Let us however briefly introduce an operation on sign-vectors needed to axiomatize AOMs:

$$
(X\oplus Y)_e:=\left\{\begin{array}{ll}
0 & \textrm{ if } e \in S(X,Y)\\
(X\circ Y)_e &\textrm{ otherwise.}
\end{array}\right.
$$

\medskip\noindent
{\bf Affinity:}
\begin{itemize}
\item[{\bf (A)}] Let $X,Y\in\covectors$ such that for all $e\in  S(X,-Y)$ and $W \in  \covectors$ with $W_e=0$ there are $f,g\in \mathcal{E}\setminus S(X,-Y)$ such that $W_f\neq(X\circ -Y)_f$ and $W_g\neq(-X\circ Y)_g$. We have $(X\oplus -Y)\circ Z\in\covectors$ for all $Z\in\covectors$.
\end{itemize}

We are now ready to define the central systems of sign-vectors of the present paper:
\begin{defi}\label{def:main}
 
 A system of sign-vectors  $(\mathcal{E},\covectors)$ is called a:
 \begin{itemize}
  \item\emph{complex of oriented matroids (COM)} if  $\covectors$ satisfies (FS) and (SE),
  \item\emph{affine oriented matroid (AOM)} if $\covectors$ satisfies (A), (FS), and (SE),
  \item\emph{oriented matroid (OM)} if $\covectors$ satisfies (Z), (FS), and (SE),
  \item\emph{lopsided system (LOP)} if $\covectors$ satisfies (IC) and (SE).
 \end{itemize}
\end{defi}

Note that for OMs one can replace (Z) and (FS) by (C) and:

\medskip\noindent
{\bf Symmetry:}
\begin{itemize}
\item[{\bf (Sym)}] $-X\in\covectors$ for all $X\in\covectors$.
\end{itemize}

%
%
%
%

Let $\covectors \subseteq  \{0,-1,1\}^\mathcal{E}$ be a system of sign-vectors and $e\in \mathcal{E}$. For $X\in \covectors$ let $X\backslash e$ be the element of $\{0,-1,1\}^{\mathcal{E}\backslash \{e \}}$ obtained by deleting the coordinate $e$ from $X$.  Define operations  $\covectors/e=\{X \backslash e \mid X\in \covectors, X_e=0\}$ as taking the \emph{hyperplane} of $e$ (usually referred to as contraction) and $\covectors\backslash e=\{ X \backslash e \mid X\in \covectors\}$ as the \emph{deletion} of $e$. A sign-system that arises by deletion and taking hyperplanes from another one is called a \emph{minor}. Furthermore denote by
$\covectors_e^+:= \{ X\backslash e\in \mid X\in \covectors, X_e=+\}$ and
$\covectors_e^-:=\{ X\backslash e\mid X\in \covectors, X_e=-\}$ the positive and negative (open) \emph{halfspaces} with respect to $e$.

A theorem due to Karlander~\cite{Kar-92} characterizes AOMs as exactly the halfspaces of OMs. However, his proof contains a flaw that has only been observed and fixed recently in~\cite{Bau-16}.

The following is easy to see, see e.g.~\cite[Lemma 2]{Ban-15}.

\begin{lem}\label{lem:commute}
For any system of sign-vectors the operations of taking halfspaces, hyperplanes and deletion commute.
\end{lem}

Our systems of sign-vectors behave well with respect to the above operations:
\begin{lem}\label{lem:signminors}
 The classes of COMs, AOMs, OMs, and LOPs are minor closed. Moreover, COMs and LOPs are closed under taking halfspaces.
\end{lem}
\begin{proof}
 The result for OMs is folklore and can be found for instance in~\cite{bj}. For COMs this was shown in~\cite[Lemma 1]{Ban-15} and~\cite[Lemma 4]{Ban-15}. For LOPs it is easy to see, that (IC) is preserved under minors and taking halfspaces. The minor-closedness of AOMs is a little more involved, when only using the axioms given above, see e.g.~\cite{Del-17} but using that they are halfspaces of OMs this follows directly from Lemma~\ref{lem:commute}.
\end{proof}

The \emph{rank} of a system of sign-vectors $(\mathcal{E},\covectors)$ is the largest integer $r$ such that there is subset $A\subseteq\mathcal{E}$ of size $|\mathcal{E}|-r$ such that $\covectors\backslash A=\{+, -, 0\}^r$. In other words, the rank of $(\mathcal{E},\covectors)$ is just the VC-dimension of $\covectors$, see~\cite{Vap-15}. Note that this definition of rank coincides with
the usual rank definition for OMs, see~\cite{daS-95}.


A system of sign-vectors $(\mathcal{E},\covectors)$ is \emph{simple} if it satisfies the following two conditions:

\begin{itemize}
\item[{\bf (N1$^*$)}]  for each $e \in \mathcal{E}$,  $\{+, -,0 \}=\{X_e: X\in \covectors\}$;
\item[{\bf (N2$^*$)}]  for each pair $e\neq f$ in $\mathcal{E}$,  there exist $X,Y \in \covectors$ with $\{X_eX_f,Y_eY_f\}=\{+, -\}$.
\end{itemize}

An element $e \in \mathcal{E}$ not satisfying (N1$^*$) is called \emph{redundant}. Note that redundant elements in OMs are zero everywhere and are called \emph{loops}, see~\cite{bjvestwhzi-93}, but in COMs also a sign can be present on a redundant element. Two elements $e,f\in\mathcal{E}$ are called \emph{parallel} if they do not satisfy (N2$^*$). Note that parallelism is an equivalence relation on $\mathcal{E}$. We denote by $\overline{e}$ the class of elements parallel to $e$, for $e\in\mathcal{E}$. The notion of parallelism coincides with the one in OMs.

For every COM $(\mathcal{E}, \covectors)$ there exists up to reorientation and relabeling of coordinates a unique simple COM, obtained by successively applying operation $\covectors\backslash e$ to the redundant coordinates $e\in \mathcal{E}$ and to elements of parallel classes with more than one element. See~\cite[Proposition 3]{Ban-15} for the details. Note that by Lemma~\ref{lem:commute} the order in which these operations are taken is irrelevant and by Lemma~\ref{lem:signminors} all the classes of systems of sign-vectors at consideration here, are closed under this operation. We will denote by $\mathcal{S}(\mathcal{E},\covectors)$ the unique \emph{simplification} of $(\mathcal{E}, \covectors)$.


\section{Systems of sign-vectors and partial cubes}\label{sec:alltogethernow}
This section finally builds the link between systems of sign-vectors and partial cubes, that in a sense is the very basis of our results. It is therefore built on properties established in Sections~\ref{sec:pcminors} and~\ref{sec:signvectors}. We begin with a kind of dictionary between axiomatic properties of systems of sign vectors and the behavior of metric subgraphs of partial cubes. In particular we will show in Theorem~\ref{thm:dictionaryOMC} that in tope graphs of COMs all antipodal subgraphs are gated and its corollaries for OMs, AOMs, and LOPs later. Recall that Figure~\ref{fig:COMtopegraphexample} shows the tope graph of a COM.

The \emph{topes} of a system of sign-vectors $(\mathcal{E},\covectors)$ are the elements of $\topes:=\covectors\cap\{+, -\}^\mathcal{E}$.
If $(\mathcal{E},\covectors)$ is simple, we define the \emph{tope graph} $G(\covectors)$ of $(\mathcal{E},\covectors)$ as the (unlabeled) subgraph of $Q_{\mathcal{\mathcal{E}}}$ induced by $\topes$.
If $(\mathcal{E}, \covectors)$ is non-simple, we consider $G(\covectors)$ as the tope graph of its simplification $\mathcal{S}(\mathcal{E},\covectors)$.

In general $G(\covectors)$ is an unlabeled graph and even though it is defined as a subgraph of a hypercube $Q_{\mathcal{\mathcal{E}}}$ it could possibly have multiple non-equivalent embeddings in $Q_{\mathcal{\mathcal{E}}}$. We call a system $(\mathcal{E},\covectors)$ a \emph{partial cube system} if its tope graph $G(\covectors)$ is an isometric subgraph of $Q_{\mathcal{\mathcal{E}}}$  in which the edges correspond to sign-vectors of $\covectors$ with a single $0$. It is well-known that partial cubes have a unique embedding in $Q_{\mathcal{\mathcal{E}}}$ up to automorphisms of $Q_{\mathcal{\mathcal{E}}}$, see e.g.~\cite[Chapter 5]{Ovc-11}. In other words, the tope graph of a simple partial cube system is invariant under reorientation. For this reason we will, possibly without an explicit note, identify vertices of a partial cube $G(\covectors)$ with subsets of $\{+, -\}^\mathcal{E}$.
The following was proved in~\cite[Proposition 2]{Ban-15}:

\begin{lem}\label{lem:OMCtoTopeGraph}
Simple COMs 
are partial cube systems.
\end{lem}

Before presenting basic results regarding partial cube systems, we discuss how the minor operations and taking halfspaces as defined in Section~\ref{sec:signvectors} affect tope graphs. So let $(\mathcal{E},\covectors)$ be a simple partial cube system.

First note that deletion does not affect the simplicity of $(\mathcal{E},\covectors)$. Furthermore, since $(\mathcal{E},\covectors)$ is a partial cube system, the tope graph $G(\covectors\backslash e)$ corresponds to $\pi_e(G(\covectors))$ obtained from $G(\covectors)$ by contracting all the edges in the $\Theta$-class corresponding to coordinate $e$, as defined in Section~\ref{sec:pcminors}. 
Noticing that $\covectors(Q_r)$, for a hypercube $Q_r$, equals $\{+, -, 0\}^r$ we immediately get the following lemma from the definition of the rank of a system of sign-vectors.

\begin{lem}\label{lem:rankchar}
The rank of a partial cube system $(\mathcal{E},\covectors)$ is the largest $r$ such that $G(\covectors)$ contracts to $Q_r$.
\end{lem}

Also, the halfspace $\covectors_e^+$ is easily seen to be simple and its tope graph corresponds to the restriction $\rho_{e^+}(G(\covectors))$ to the positive halfspace of $E_e$ for $e\in\mathcal{E}$, as defined in Section~\ref{sec:pcminors}.

The hyperplane $\covectors/e$ does not need to be a simple system of sign-vectors nor a partial cube system. However, we can establish the following:

\begin{lem}\label{lem:hyperplane=hyperplane}
 Let $(\mathcal{E},\covectors)$ be a partial cube system and $e\in\mathcal{E}$. If $\zeta_e(G(\covectors))$ is a well-embedded partial cube, then $\zeta_e(G(\covectors))\cong G(\covectors/e)$.
\end{lem}
\begin{proof}
Clearly, both sets of vertices correspond to the set of edges of $G(\covectors)$. If there is an edge in $G(\covectors/e)$, it corresponds to two edges of $G(\covectors)$ in $E_e$ such that the $\Theta$-classes different than $E_e$ crossing the interval between the two edges form parallel elements of  $G(\covectors/e)$. This means that the interval does not cross $E_e$ in other elements besides the two edges. The interval must include a convex traverse between the two edges, but since there is no other edge in $E_e$ in it, the traverse is a convex cycle. This implies that there is a corresponding edge in $\zeta_e(G(\covectors))$.

Now, let $\zeta_e(G(\covectors))$ be a well-embedded partial cube. An edge of it corresponds to a convex cycle $C$. By Lemma~\ref{lem:pchyperplane} all cycles crossing $E_e$ and another $\Theta$-class from $C$ cross all its $\Theta$-classes. By Lemma~\ref{lem:parallel}, the corresponding elements in $\covectors/e$ are parallel. Therefore the edge corresponds to an edge of $G(\covectors/e)$.
\end{proof}


%

%
%
%

The correspondences before the lemma in particular give that deletions and halfspaces of partial cube systems coincide with pc-minors, which together with Lemma~\ref{lem:signminors} gives:
\begin{prop}\label{prop:minorsandrank}
 Let $(\mathcal{E},\covectors)$ and $(\mathcal{E}',\covectors')$ be simple partial cube systems with tope graphs $G(\covectors)$ and $G(\covectors')$, respectively. If $(\mathcal{E}',\covectors')$ arises from $(\mathcal{E},\covectors)$ by deletion and taking halfspaces, then $G(\covectors')$ is a pc-minor of $G(\covectors)$. Moreover, the families of tope graphs of COMs and LOPs are pc-minor closed.
\end{prop}

In the following, we will describe further how pc-minors and equivalently deletions and halfspaces of partial cube systems translate metric graph properties as introduced in Section~\ref{sec:pcminors} into properties of sign-vectors.

 For $X\in\covectors$ we set $\topes(X):=\{T\in \topes\mid X\circ T=T\}$ and denote by $G(X)$ the subgraph of $G(\covectors)$ induced by $\topes(X)$. Note that in OMs the set $\topes(X)$ is sometimes denoted as $\mathrm{star}(X)$, see~\cite{bjvestwhzi-93}. Furthermore, let $\mathcal{H}(\covectors)=\{G(X)\mid X\in\covectors\}$ be the set of subgraphs of $G(\covectors)$ obtained by considering $G(X)$ for all $X\in\covectors$. Conversely, given a convex subgraph $G'$ of a partial cube $G$ with $\Theta$-classes $\mathcal{E}$ denote by $\chi(G')\in\{+,-,0\}^{\mathcal{E}}$ the sign-vector defined by setting for $e\in\mathcal{E}$: $$\chi(G')_e=\begin{cases}
 + & \text{if } G'\subseteq E_e^{+},\\                                                                                                                                                                                                                                                                                                                                                                                                                                               
 - & \text{if } G'\subseteq E_e^{-},\\
 0 & \text{otherwise.}
 \end{cases}$$
Note that for each vertex $v\in G(\covectors)$, $\chi(v)=v$.
Furthermore, let $\covectors(\mathcal{H})=\{\chi(G')\mid G'\in \mathcal{H}\}$ for a set $\mathcal{H}$ of convex subgraphs of $G$.

\begin{prop}\label{prop:dictionarybasic}
 In a simple partial cube system $(\mathcal{E},\covectors)$ for each $X\in\covectors$ its tope-graph $G(X)$ is a convex subgraph of $G(\covectors)$. Conversely, if $G=(V,E)$ is a partial cube and  $\mathcal{H}$ a set of convex subgraphs of $G$, such that $\mathcal{H}$ includes all the vertices of $G$, then there is a simple $(\mathcal{E},\covectors)$ such that $G=G(\covectors)$ and $\mathcal{H}=\mathcal{H}(\covectors)$.
\end{prop}
\begin{proof}
 Let $X\in\covectors$. Since $G(\covectors)$ is an isometric subgraph of $Q_{\mathcal{\mathcal{E}}}$, each $e\in \mathcal{E}$ can be identified with a $\Theta$-class $E_e$ for $e\in\mathcal{E}$ such that the halfspaces $E_e^+$ and $E_e^-$ of $G(\covectors)\setminus E_e$ are convex. Now, $\topes(X)$ induces the subgraph $\bigcap_{e\in\underline{X}}E_e^{X_e}$, i.e.~is a restriction of $G(\covectors)$ and therefore is convex.

 For the converse, by Lemma~\ref{lem:restriction} any convex subgraph $G'\in\mathcal{H}$ may be described as the intersection of convex halfspaces $E_e^{+, -}$ for some $E_e$ with $e\in\mathcal{E}$. This allows for a correspondence of convex graphs $G'\in\mathcal{H}$ and sign-vector via $\chi(G')\in\{+, -,0\}^\mathcal{E}$ and establishes $\mathcal{H}=\mathcal{H}(\covectors)$. Since every vertex is contained in $\mathcal{H}$, we have $G=G(\covectors)$.
\end{proof}



The following establishes a connection between the gates of a convex set and the composition operator. The statement specialized to tope graphs of OMs can be found in~\cite[Exercise 4.10]{bjvestwhzi-93}.

\begin{lem}\label{lem:dictionarygated}
Let $G$ be a partial cube embedded in a hypercube, $G'$ a convex subgraph of $G$ and $v$ a vertex of $G$. Then $w$ is the gate for $v$ in $G'$ if and only if $\chi(w)=\chi(G')\circ \chi(v)$. Therefore, a subgraph $G'$ is gated if and only if for all $v\in G$ there is a $w\in G$ such that $\chi(G')\circ \chi(v)=\chi(w)$.
\end{lem}
\begin{proof}
 First note that $\chi(w)=\chi(G')\circ \chi(v) \Leftrightarrow S(\chi(v),\chi(w))=S(\chi(v),\chi(G'))$. Thus, if $\chi(w)=\chi(G')\circ \chi(v)$ then $S(\chi(v),\chi(w))=S(\chi(v),\chi(G'))$. Hence, the concatenation of a shortest $(v,w)$-path and a shortest $(w,w')$-path for $w'\in G'$ does not cross any $\Theta$-class twice, since by convexity of $G'$ the $\Theta$-classes crossed by the second part all are from $\chi(G')^0$. Hence it is a shortest path.

 Now, let $v\in G$ and $w$ a gate for $v$ in $G'$. Suppose $e\in S(\chi(v),\chi(w))\setminus S(\chi(v),\chi(G'))$. This means a $\Theta$-class $E_e$ for $e\in\mathcal{E}$ splits $G$ and also $G'$ into two halfspaces, but $v$ and $w$ do not lie on the same side. This is $w$ cannot lie on a shortest path from $v$ to a vertex in the halfspace of $G'$ not containing $w$. This contradicts that $w$ is a gate for $v$ in $G'$ and hence $S(\chi(v),\chi(w))=S(\chi(v),\chi(G'))$, which implies $\chi(w)=\chi(G')\circ \chi(v)$.
\end{proof}

\begin{lem}\label{lem:antipodes}
In an antipodal partial cube $G$, the \emph{antipodal mapping} $v\mapsto -v$ is a graph automorphism and for every convex subgraph $-\chi(G')=\chi(-G')$.
\end{lem}
\begin{proof}
Since every vertex has a unique antipode, $v\mapsto -v$ is indeed a mapping. Since it is an involution it is bijective. To show that it is indeed a homomorphism let $vw$ be an edge of $G$. Since $G$ is a partial cube and since $\conv(v,-v)=G$ there is a shortest $(v,-v)$-path starting with the edge $vw$. Analogously, there is a shortest $(w,-w)$-path starting with the edge $vw$. Say this edge is in $\Theta$-class $E_e$. Now, $S(\chi(v),\chi(-v))\setminus S(\chi(v),\chi(-w))=e$ and hence $S(\chi(-v),\chi(-w))=e$, i.e.~$-v-w$ is an edge.

Now, let $G'$ be a convex subgraph of $G$. Since $\chi(G')$ just records the signs of the halfspaces of the classes of $\mathcal{E}$ which contain $G'$ and the antipodal mapping sends each vertex $v$ to the vertex sitting on all the other sides, we obtain the result.
\end{proof}

For a partial cube $G$ isometrically embedded in a hypercube  $Q_\mathcal{E}$ 
define the set of sign vectors $$\covectors(G)=\{X\in \{0,+, -\}^\mathcal{E}\mid \textrm{ for all } v\in G\textrm{ there exists } w\in G :X\circ (-\chi(v))=\chi(w)\}.$$


\begin{lem}\label{lem:L(G)}
 Let $G$ be a partial cube isometrically embedded in a hypercube  $Q_\mathcal{E}$. Then $\covectors(G)$ is a partial cube system that satisfies (FS) (and therefore (C)) and the set $\mathcal{H}(\covectors(G))$ of corresponding subgraphs coincides with the antipodal gated subgraphs of $G$.
\end{lem}
\begin{proof}
 Clearly, the topes of $\covectors(G)$ correspond to the vertices of $G$. Moreover, a sign-vector with a single 0-entry is in $\covectors(G)$ if and only if both of the possibly signings of that entry are topes of $\covectors(G)$.
 
 Since (FS) implies (C) (also when restricted to topes) we have that $X\in\covectors(G)$ implies $X\circ \chi(v)\in \covectors(G)$ for all $v\in G$. To show (FS) for $\covectors(G)$ let $X,Y\in\covectors(G)$ and $v\in G$ and note that $(X\circ(-Y))\circ(-\chi(v))=X\circ(-(Y\circ \chi(v)))$. Since by (C) we have that $Y\circ \chi(v)\in\covectors(G)$ it follows that $X\circ(-Y)\in\covectors(G)$.
 
 We prove now the second part of the statement. Let $X\in \covectors(G)$ and a vertex $v\in G(X)$. We have that $X\circ -\chi(v)\in \topes(X)$ is the antipode of $v$ in $G(X)$. This is, $G(X)$ is antipodal. Furthermore, since (FS) implies (C) we also have that $X\circ \chi(v)\in \topes(X)$, for all $v\in G$. By Lemma~\ref{lem:dictionarygated} we have that $G(X)$ is gated.
 
 Conversely, if $A$ is an antipodal gated subgraph of $G$, then by Lemma~\ref{lem:dictionarygated}, we have that for the gate $v'$ of $v$ in $A$ it holds $\chi(A)\circ \chi(v)=\chi(v')$. Now, the antipode of the gate of $v'$ in $A$ has to correspond to $\chi(A)\circ -\chi(v)$. Thus, $\chi(A)\in\covectors(G)$.
 
%
%
%
%
\end{proof}

%
%

%


Proposition~\ref{prop:dictionarybasic} states that in a simple system of sign-vectors there is a correspondence between its vectors and a subset of the set of convex subgraphs of its tope graph. The following proposition determines which convex subgraphs are in the subset if the system is a COM. For its statement denote by $\topegraphs$ the class of tope graphs of COMs. Moreover, we call a partial cube \emph{antipodally gated} if all its antipodal subgraphs are gated and denote their class by $\AG$.

\begin{thm} \label{thm:dictionaryOMC}
For a simple COM $(\mathcal{E},\covectors)$ with embedded tope graph $G$ we have

							  \begin{center}

		  $\begin{array}{rl}
                                                           \covectors & =\{\chi(G')\mid G'\text{ antipodal subgraph of }G(\covectors)\}\\
                                                           &=\{\chi(G')\mid G'\text{ antipodal gated subgraph of }G(\covectors)\}\\
                                                           &=\covectors(G).
                                                          \end{array}$\end{center}
In particular, tope graphs of COMs are antipodally gated, i.e, $\topegraphs\subseteq \AG$.
\end{thm}
\begin{proof}
Trivially, we have $$\{\chi(G')\mid G'\text{ antipodal gated subgraph of }G(\covectors)\}\subseteq\{\chi(G')\mid G'\text{ antipodal subgraph of }G(\covectors)\}.$$ The equality
$$\covectors(G)= \{\chi(G')\mid G'\text{ antipodal gated subgraph of }G(\covectors)\}$$ is precisely Lemma~\ref{lem:L(G)}, while
%
the inclusion $\covectors\subseteq \covectors(G)$ follows immediately from (FS).

Thus, we end by proving $\{\chi(G')\mid G'\text{ antipodal subgraph of }G(\covectors)\}\subseteq \covectors$.
Let $G'$ be an antipodal subgraph of $G(\covectors)$. We show $\chi(G')\in\covectors$ by induction on $|\chi(G')^0|$. If $|\chi(G')^0|=0$ then $\chi(G')\in\topes$ and we are done. If $|\chi(G')^0|>0$ then take a maximal proper antipodal subgraph $G''$ of $G'$. Since vertices are antipodal subgraphs, $G''$ exists. By induction hypothesis $\chi(G'')\in\covectors$. Define $G'''$ to be the subgraph of $G'$ induced by all antipodes of vertices of $G''$ with respect to $G'$. By Lemma~\ref{lem:antipodes}, we have $G'''\cong G''$. Moreover, $\chi(G'')=\chi(G')\circ-\chi(G''')$. 
Since $G''$ is a proper antipodal subgraph of $G'$, there exists an $e\in \underline{\chi(G'')}\cap \chi(G')^0$. Then $e\in S(\chi(G''),\chi(G'''))$ and 
we now can apply (SE) to $\chi(G'')$ and $\chi(G''')$ with respect to $e$.  
We obtain $Z \in  \covectors$ such that
$Z_e=0$  and  $Z_f=(\chi(G'')\circ \chi(G'''))_f$  for all $f\in \mathcal{E}\setminus S(\chi(G''),\chi(G'''))$. By the inclusions that we have shown in the first two parts of the proof, $Z$ corresponds to an antipodal subgraph $G(Z)$. By (SE), $G(Z)$ strictly contains $G''$ and is contained in $G'$. By the maximality of $G''$ we have $G(Z)=G'$ and therefore $\chi(G')=Z\in\covectors$.
\end{proof}

As a consequence of Theorem~\ref{thm:dictionaryOMC} we immediately get:

\begin{cor}\label{cor:determined}
 Every simple COM is uniquely determined by its tope set and up to reorientation by its tope graph.
\end{cor}
 Corollary~\ref{cor:determined} had only been proved in a non-constructive way, see~\cite[Propositions 1 \& 3]{Ban-15}. The constructive statement here is in fact a generalization of a theorem known for OMs, usually attributed to Mandel, see~\cite{Cor-85} and the fact that an OM is determined up to reorientation by its tope graph is due to~\cite{Bjo-90}. 
 
 The following justifies that we can restrict ourselves to simple COMs when studying tope graphs.
\begin{lem}\label{lem:covectors(G)}
A partial cube $G$ is in $\topegraphs$ if and only if $\covectors(G)$ is a simple COM.
\end{lem}

\begin{proof}
If $\covectors(G)$ is a simple COM, then $G$ is by definition its tope graph, thus $G$ is in $\topegraphs$. On the other hand, let $G$ be a graph in $\topegraphs$, and $\covectors$ its up to reorientation unique simple COM. By Theorem~\ref{thm:dictionaryOMC}, $\covectors = \{X\in\{0, -1, 1 \}^{\mathcal{E}}\mid X\circ -T\in\topes \text{ for all } T\in\topes\} = \covectors(G)$.
%
%
%
%
\end{proof}

Th following will be essential for the Handa-type characterization of tope graphs of COMs in Corollary~\ref{cor:Handa}. It can be seen as the graph theoretical analogue of the fact the COMs are closed under taking hyperplanes.
\begin{lem}\label{lem:COMhyperplane}
 In a COM $(\mathcal{E},\covectors)$ all zone graphs $\zeta_f(G(\covectors))$ are well-embedded partial cubes. In particular, $\topegraphs$ is closed under taking zone graphs.
\end{lem}
\begin{proof}
 Suppose there are two convex cycles $C,C'$ in $G(\covectors)$ that contradict Lemma~\ref{lem:pchyperplane}, i.e., both are crossed by $E_f$ and $E_g$ and $C$ is crossed by $E_h$ but $C'$ is not. By the second equivalence in Theorem~\ref{thm:dictionaryOMC}, the cycles are gated. Without loss of generality assume that $C'$ is completely in $E_h^+$. On the other hand, we can reorient $E_f$ and $E_g$ in a way that the vertices on $C\cap E_h^+$ are in $E_f^+\cap E_g^+, E_f^+\cap E_g^-$ and $E_f^-\cap E_g^-$. But then any vertex in $C'$ that is in $E_f^- \cap E_g^+$ has no gate to $C$, a contradiction. 
 
 Now, since COM is closed under taking hyperplanes by Lemma~\ref{lem:signminors}, Lemma~\ref{lem:hyperplane=hyperplane} gives that $\topegraphs$ is closed under taking zone graphs.
 
%
%
%
\end{proof}

Let us finally describe how tope graphs of the other systems of sign-vectors from Section~\ref{sec:signvectors} specialize tope graphs of COMs.  We will denote the classes of tope graphs of OMs, AOMs, and LOPs by $\topegraphsOM$, $\topegraphsAOM$, and $\topegraphsLOP$. The following three propositions will be cast in the proofs of the characterizations of $\topegraphsOM$, $\topegraphsAOM$, and $\topegraphsLOP$, see Corollaries~\ref{cor:OM},~\ref{cor:AOM}, and~\ref{cor:LOP}.

A consequence of Lemma~\ref{lem:antipodes} is:
\begin{prop} \label{prop:OMandCOM}
 A graph is in $\topegraphsOM$ if and only if it is antipodal and in $\topegraphs$.
\end{prop}

A not yet intrinsic description of tope graphs of AOMs follows:
\begin{prop} \label{prop:affOMandOM}
 A graph is in $\topegraphsAOM$ if and only if it is a halfspace of a graph in $\topegraphsOM$.
\end{prop}

Interpreting axiom (IC) in the partial cube model we also get:
\begin{prop} \label{prop:LOPandCOM}
 A graph is in $\topegraphsLOP$ if and only if all its antipodal subgraphs are hypercubes and it is in $\topegraphs$.
\end{prop}

\section{The excluded pc-minors}\label{sec:forbidden}
In the present section we will introduce the set $\Q$ of minimal excluded pc-minors for tope graphs of COMs. After providing several properties with respect to pc-minors and zone graphs, we give a couple of methods of detecting a member of $\Q$ in a given graph.  The main result of the section is that partial cubes excluding $\Q$ are topes graphs of COMs (Theorem~\ref{thm:topegraphs}).
The proofs in this section use properties of zone graphs established in Section~\ref{sec:pcminors}, minor-closedness of COMs seen in Section~\ref{sec:signvectors}, as well as properties of tope graphs of COMs shown in Section~\ref{sec:alltogethernow}

Let $Q_n$ be the hypercube, $v\in Q_n$ any of its vertices and $-v$ its antipode. Let $Q^-_n:=Q_n\setminus -v$ be the hypercube minus one vertex. Consider the set of partial cubes arising from $Q_n^-$ by deleting any subset of $N(v)\cup\{v\}$.
It is easy to see that if $n\geq 4$ a graph obtained this way from $Q_n^-$ is a partial cube unless $v$ is not deleted but at least two of its neighbors are deleted. Denote by $Q_n^{-*}$ the partial cube obtained by deleting exactly one neighbor of $v$, and by $Q_n^{--}(m)$ the graph obtained by deleting $v$ and $m$ neighbors of $v$, respectively, where for $Q_n^{--}(0)$ we sometimes simply write $Q_n^{--}$. It is easy to see that $Q_n^-$ and $Q_n^{--}$ are tope graphs of (realizable) COMs. For $n\leq 3$ all the partial cubes arising by the above procedure are isomorphic to $Q_n^-$ or $Q_n^{--}$, thus the interesting graphs appear for $n\geq 4$. Denote their collection by $\Q=\{Q_n^{-*},Q_n^{--}(m)\mid 4\leq n; 1\leq m\leq n\}$. 

\begin{figure}[htb]
\centering
\includegraphics[width=\textwidth]{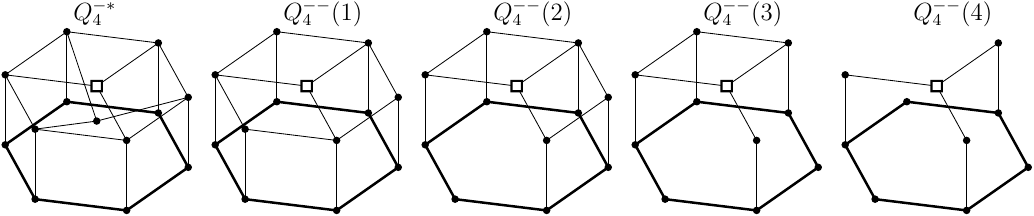}
\caption{Graphs $Q_4^{-*},Q_4^{--}(m)$, for $1\leq m \leq 4$. The 
square vertex has no gate in the bold $C_6$.}
\label{fig:Q4s}
\end{figure}


\begin{lem}\label{lem:minfamily}
The set $\Q$ is pc-minor minimal, i.e.~any pc-minor of a graph in $\Q$ is not in $\Q$. Furthermore, any graph in  $\Q$ contains an antipodal subgraph that is not gated, i.e.~$\Q\subseteq\overline{\mathrm{\AG}}$, where $\overline{\mathrm{\AG}}$ denotes the complementary class of $\AG$.
\end{lem}

\begin{proof}
%

For any $G\in\Q$ a contraction or a restriction of it is a graph isomorphic to a hypercube, a hypercube minus a vertex, or a hypercube minus two antipodal vertices. All pc-minors of these graphs are isomorphic to a hypercube or a hypercube minus a vertex. Thus, no proper pc-minor of $G$ is in $\Q$.

To see that any graph in  $\Q$ contains an antipodal subgraph that is not gated, let $G\in\Q$ and let $w$ be a neighbor of $v$ that was deleted from $Q_n^-$ to obtain $G$. The convex subgraph $A\subseteq G$ obtained from restricting $Q_n^-$ to the halfspace containing $w$ and not $v$ is isomorphic to $Q_{n-1}^{- -}$. In particular, $A$ is antipodal. But $A$ is not gated, since the neighbor $u$ of $-v$ in $G\setminus A$ has no gate in $A$. In fact, in can be seen by Lemma~\ref{lem:dictionarygated}, that the gate of $u$, if existent, must be of the form $\chi(A)\circ \chi(u)=\chi(-v)$ which is not in $G$.
\end{proof}

%
%
The following lemma will be useful for detecting pc-minors from $\Q$. We define $H$ to be the \emph{full subdivision} of a graph $G$ if every edge of it is replaced by a path of length $2$ to obtain $H$. The vertices of $H$ that correspond to vertices of $G$ are called its \emph{original vertices}.

\begin{lem}\label{lem:K_n}
Let $G$ be a partial cube and $H$ an isometric subgraph isomorphic to a full subdivision of $K_m$ such that
\begin{itemize}
\item no vertex of $G$ is adjacent to all the original vertices of $H$,
\item the convex hull of $H$ is neither isomorphic to $Q_m^-$ nor $Q_m^{--}$,
\item $H$ is inclusion minimal with this properties.
\end{itemize}
Then the convex hull of $H$ is in $\Q$.
\end{lem}

\begin{proof}
Let $H$ be as in the lemma. First note that $m\geq 4$ since for $m=2$, $H$ is isomorphic to $P_3 = Q_2^-$, and for $m=3$, $H$ is isomorphic to $C_6$. Taking into account that there is no vertex of $G$ adjacent to all the original vertices of $H$, the convex hull in both cases is isomorphic to $Q_m^-$ or to $Q_m^{--}$.

Now consider $H$ embedded in a hypercube $Q_m$ with $m\geq 4$. Since $H$ is an isometric subgraph of $G$, it is a partial cube and thus has up to reorientation a unique embedding in $Q_m$. Note that one (and therefore the only) possible embedding, is such that there exists a vertex $v$ in $Q_m$ adjacent to exactly the original vertices of $H$. Indeed, embed the original vertices as vectors with precisely one $+$ and the subdivision vertices as vectors with two $+$-signs. Then, $v$ can be chosen to be the vector consisting only of $-$. By the assumption, $v$ is not in $G$. Let $\{v_1,\ldots, v_{m-1}\}$ be any subset of the original vertices of $H$ of size $m-1$. Since $H$ is inclusion minimal, the convex hull of $\{v_1,\ldots, v_{m-1}\}$ is a graph isomorphic to $Q_{m-1}^-$ or $Q_{m-1}^{--}$ in which $v$ is a missing vertex. Note that the antipode of $v$ in $Q_m$ is at distance $m$ from $v$. We have proved that all the vertices of $Q_m$ at distance at most $m-2$ from $v$ are in $G$ while some of the neighbors of the antipode of $v$ in $Q_m$ are possibly not in $G$. Thus the convex hull of $H$ is a graph in $\Q$.
\end{proof}

A second useful lemma, tells us how to find excluded pc-minors in non-antipodal graphs with antipodal contractions.

\begin{lem}\label{lem:allcontractionsantipodal}
 Let $G$ be a partial cube with $v\in G$ such that $-v\notin G$, and $E_{e_1},\ldots, E_{e_k}$ be $\Theta$-classes of $G$ such that $\pi_{e_i}(G)$ is antipodal for all $1\leq i\leq k$. Then either $G$ contains a convex subgraph from  $\Q$ or $G$ contains a convex $Q_k^-$ or $Q_k^{--}$ crossed by precisely $E_{e_1},\ldots, E_{e_k}$ such that $-v$ is a missing vertex of it. In particular, if the latter holds and $E_{e_1},\ldots, E_{e_k}$ are all the $\Theta$-classes of $G$, then $G$ is isomorphic to $Q_k^-$.
\end{lem}

\begin{proof}
Let $v\in G$ be a vertex without an antipode in $G$. Let $E_{e_1},\ldots, E_{e_k}$ be the $\Theta$-classes as above. Every contraction $\pi_{e_i}(G)$ is antipodal, therefore $\pi_{e_i}(v)$ has an antipode in $\pi_{e_i}(G)$. 
Let $v_i\in G$ be the preimage of the antipode $-_{\pi_{e_i}(G)}\pi_{e_i}(v)$ of $\pi_{e_i}(v)$ in $\pi_{e_i}(G)$ for $1\leq i\leq k$, respectively. By definition, $v_1,\ldots,v_k$ are pairwise at distance 2, thus a part of an isometric subgraph isomorphic to the subdivision of $K_k$. Moreover, there is no vertex adjacent to all of them, since $v$ has no antipode in $G$. By Lemma~\ref{lem:K_n}, the convex hull of $v_1,\ldots,v_k$ is isomorphic to one of $Q_k^-, Q_k^{--}$, or there is an inclusion minimal subdivided $K_{k'}$ for $k'\leq k$ whose convex closure is not isomorphic to $Q_{k'}^-$ or $Q_{k'}^{--}$.
In the former case we are done while in the latter case the convex closure of the inclusion minimal subdivided $K_{k'}$ is in $\Q$. 

Assuming that $E_{e_1},\ldots, E_{e_k}$ are all the classes of $G$, the convex hull of $v_1,\ldots,v_k$ is crossed by all the $\Theta$-classes of $G$, hence it is $G$. Then $G$ is isomorphic to $Q_k^-$ or $Q_k^{--}$ if it has no convex subgraph in $\Q$. Since $G$ is not antipodal, it must be isomorphic to $Q_k^-$.
\end{proof}

\begin{lem}\label{lem:hyperplaneforbidden}
Let $G$ be a partial cube. If a zone graph $\zeta_e(G)$ is not a well-embedded partial cube, then $G$ has a pc-minor in $\{Q_4^{-*}, Q_4^{--}(m) \mid 1\leq m \leq 4\}$. 
\end{lem}
\begin{proof}
For the sake of contradiction assume $G$ does not satisfy the conditions of Lemma~\ref{lem:pchyperplane}, i.e., there are convex cycles $C_1,C_2$
both crossed by $E_e$ and $E_f$ and $C_1$ is crossed by $E_g$ but $C_2$ is not. Without loss of generality assume that $C_2$ is completely in $E_g^+$. On the other hand, we can reorient $E_e$ and $E_f$ in a way that the vertices on $C_1\cap E_g^+$ are in $E_e^+\cap E_f^+, E_e^+\cap E_f^-$ and $E_e^-\cap E_f^-$. But then any vertex $v$ in $C_2$ that is in $E_e^- \cap E_f^+$ has no gate to $C_1$.

%
%

We will now see, that this leads to the existence of a pc-minor in $\{Q_4^{-*}\}\cup \{Q_4^{--}(m) \mid 1\leq m \leq 4\}$.
First contract any $\Theta$-class different from $E_e,E_f,E_g$ but crossing $C_1$, obtaining graph $G'$. Then $C_1$ is contracted to a convex 6-cycle $C_1'$ in $G'$, by Lemma~\ref{convex_hull}. It is non-gated since the image of $v$ in $G'$ still is in $E_e^-\cap E_f^+\cap E_g^+$ while $E_e^-\cap E_f^+\cap E_g^+\cap C'_1=\emptyset$. Now consider a maximal sequence $S$ of contractions of $\Theta$-classes  different from $E_e,E_f,E_g$ such that for the image $C_1''$ of $C_1'$ we have that $\conv(C_1'')$ is a $6$-cycle. Since $v$ has no gate to $C_1'$, contracting all the  $\Theta$-classes different from $E_e,E_f,E_g$ maps $v$ to a vertex in $\conv(C_1'')\setminus C_1''$ contradicting that $\conv(C_1'')$ is a $6$-cycle. Thus $S$ is not equal to all the $\Theta$-classes different from $E_e,E_f,E_g$.

Pick any $\Theta$-class $E_h$ not in $S\cup \{E_e,E_f,E_g\}$. By maximality  $\conv(\pi_h(C_1''))$ is not a $6$-cycle thus it must be isomorphic to $Q_3^-$ or $Q_3$. Let $u$ be a vertex in $Q_3^-$ or $Q_3$ adjacent to three vertices $\pi_h(u_1),\pi_h(u_2),\pi_h(u_3)\in\pi_h(C_1'')$ with $u_i\in C_1''$ for $i\in \{1,2,3\}$. No preimage $u'$ of $u$ is adjacent to any $u_i$, for $i\in \{1,2,3\}$, since otherwise $u'$ would be in $\conv(C_1'')$. Thus, $u', u_1,u_2,u_3$ are pairwise at distance two. Since $\conv(C_1'')$ is a $6$-cycle, there is no vertex adjacent to all of them. Together with their connecting 2-paths they form an isometric $K_4$. Moreover, their convex hull is not isomorphic to $Q_4^-$ or $Q_4^{--}$ since such graphs have no convex 6-cycles. By Lemma~\ref{lem:K_n}, $G$ has a pc-minor in $\{Q_4^{-*},Q_4^{--}(m) \mid 1\leq m \leq 4\}$.
%
%
%
%
\end{proof}

\begin{lem}\label{lem:hyperplanesofQ}
 If $G\in\Q$, then there is a sequence $(e_1,\ldots,e_k)$ of $\Theta$-classes such that $\zeta_{e_1,\ldots,e_k}(G)$ is not a partial cube.
\end{lem}
\begin{proof}
 Let $G\in\{Q_n^{-*},Q_n^{--}(m)\mid 1\leq m\leq n\}$ for some $n\geq 4$. Let $v\in Q_n$ the vertex from the definition of $\Q$.
 
 If $n=4$ one can easily see in Figure \ref{fig:Q4s} that $G$ has a zone graph containing a $C_5$ if $G\in\{Q_4^{-*},Q_4^{--}(1),Q_4^{--}(2)\}$ or a $C_3$ if $G\in\{Q_4^{--}(3),Q_4^{--}(4)\}$ i.e., the zone graph is not a partial cube.
 
%
 
 Let now $n>4$. If $G \in \{Q_n^{-*}, Q_{n}^{--}(m) \mid 1\leq m \leq n-1\}$, let $w$ be one of the neighbors of $v$ that is in $G$ and let $e$ be the $\Theta$-class of $G$ coming from the edge $vw$ in $Q_n$. It is easy to see that $\zeta_e(G) \in \{Q_{n-1}^{-*}, Q_{n-1}^{--}(m) \mid 1\leq m \leq n-1\}$. If otherwise $G = Q_n^{--}(n)$, then let $w$ be any of the neighbors of $v$. Then $w$ is missing in $G$ and let $e$ be the $\Theta$-class of $G$ coming from the edge $vw$ in $Q_n$. One can check that $\zeta_e(G) = Q_{n-1}^{--}(n-1)$.
The lemma follows by iterating the given zone graphs until arriving at a non partial cube.
\end{proof}

A useful sort of converse of the proof of Lemma~\ref{lem:hyperplanesofQ} is the following: 

\begin{lem}\label{lem:hyperplanesinQ}
 If $G$ is partial cube and $\zeta_e(G)\in\Q$ for some $\Theta$-class $e$, then $G$ has a pc-minor in $\Q$.
\end{lem}
\begin{proof}
Assume that $\zeta_e(G)$ is well-embedded since otherwise $G$ has a pc-minor in $\Q$, by Lemma \ref{lem:hyperplaneforbidden}. Let $G$ be pc-minor minimal, without affecting $H:=\zeta_e(G)$ and let $N$ be the number of $\Theta$-classes of $H$. Then by minimality $G$ is the convex hull of $E_e$ and has $N+1$ $\Theta$-classes. There exist two isometric copies of $H$ in $B$ with edges of $E_e$ being a matching of them. Let $v_1,\ldots,v_N$ be the original vertices of the subdivided $K_N$ in the first copy and $v'_1,\ldots,v'_N$ the original vertices of the subdivided $K_N$ in the second copy. Not both subdivisions can have a vertex adjacent to all of the original vertices of the subdivisions since then there would be an edge in $E_e$ between the two vertices, which is not true by the definition of $H$. Without loss of generality, assume that $v_1,\ldots,v_N$ have no common neighbor. Then by Lemma~\ref{lem:K_n}, either $G$ has a pc-minor in $\Q$ and we are done, or the convex hull of $v_1,\ldots,v_N$ is isomorphic to $Q_{N}^-$ or $Q_N^{--}$. 
 Analogously, if $v'_1,\ldots,v'_N$ have no common neighbor their convex hull is isomorphic to $Q_{N}^-$ or $Q_N^{--}$. But then $H$ is isomorphic to $Q_{N}^-$ or $Q_N^{--}$ which is not the case. Thus $v'_1,\ldots,v'_N$ have a common neighbor, say $u$. Then $u,v_1,\ldots,v_N$ are pairwise at distance 2, without a common neighbor. Then by Lemma~\ref{lem:K_n}, either $G$ has a pc-minor in $\Q$ or their convex hull $H$ is isomorphic to $Q_{N+1}^-$ or $Q_{N+1}^{--}$. The latter cannot be since then $H$ is not in $\Q$. Therefore, $G$ has the pc-minor in $\Q$.
\end{proof}

%
%

We are ready to prove the main theorem of this section.

%
%
%

\begin{thm}\label{thm:topegraphs}
A partial cube that has no pc-minor from the set $\Q$  is the tope graph of a COM, i.e.,  $\mathcal{F}(\Q)\subseteq \topegraphs$.
\end{thm}

\begin{proof}
For the sake of contradiction, assume that we can pick $G$, a smallest graph that is not in $\topegraphs$ but in $\mathcal{F}(\Q)$. In particular, since $\mathcal{F}(\Q)$ is pc-minor closed, every pc-minor of $G$ is in $\topegraphs$. 

By Lemma~\ref{lem:covectors(G)}, $G$ is a graph in $\topegraphs$ if and only if $\covectors(G)$ is a COM. 
Since $G$ is not in $\topegraphs$, but $G$ is a partial cube and by Lemma~\ref{lem:L(G)} $\covectors(G)$ is a partial cube system that satisfies (C),~\cite[Theorem 3]{Ban-15} gives that $\covectors(G)$ has a hyperplane $\covectors(G)/e$ that is not a COM. 

Since $G\in\mathcal{F}(\Q)$, by Lemma~\ref{lem:hyperplaneforbidden} we have that $G':=\zeta_e(G)$ is a well-embedded partial cube.
By Lemma~\ref{lem:hyperplane=hyperplane} we get $G' \cong G(\covectors(G)/e)$, i.e., it is the tope graph of the hyperplane.


\begin{claim}
We have $\mathcal{S}(\covectors(G)/e)=\covectors(G')$.
\end{claim}

\begin{proof}
By Lemma~\ref{lem:L(G)}, the elements of $\covectors(G')$ correspond to antipodal subgraphs of $G'$. Furthermore, it is not hard to see that elements of $\mathcal{S}(\covectors(G)/e)$ correspond to antipodal subgraphs of $G$ that are crossed by $E_e$, where redundant coordinates have been deleted. If $A$ is an antipodal subgraph of $G$ crossed by $E_e$, then by Lemma \ref{lem:antipodes}, each edge $uv\in E_e$ of $A$ has an antipodal edge $-_Au-_Av\in E_e$. Thus the zone graph of $A$ corresponding to $E_e$ is an antipodal subgraph of $G'$ and we get that $\mathcal{S}(\covectors(G)/e) \subseteq \covectors(G')$.

Conversely, assume that there is an antipodal graph $A'$ in $G'$ that does not correspond to a zone graph of an antipodal subgraph of $G$. By definition of $G'$ we can identify its vertices with edges of $G$ in $E_e$. Let $A$ be the convex hull of those edges in $G$ that correspond to vertices of $A'$, and let $E_e,E_{e_1},\ldots,E_{e_k}$ be the $\Theta$-classes crossing $A$. Since $A'$ does not correspond to a zone graph of an antipodal graph, $A$ is not antipodal.

By minimality of $G$, for every $E_{e_j}$ the contraction $\pi_{e_j}(G)$ is the tope graph of a COM. This is, $\covectors( \pi_{e_j}(G)) = \covectors(G) \backslash e_j$ is a COM. Hence, by Lemma~\ref{lem:signminors} its hyperplane $(\covectors(G) \backslash e_j)/ e$ is a COM as well. Now Lemma~\ref{lem:commute} gives
$\mathcal{S}((\covectors(G) \backslash e_j)/ e) = \mathcal{S}(\covectors(G)/e)\backslash e_j$. We have proved that $\mathcal{S}(\covectors(G)/e)\backslash e_j$ is a COM for every $E_{e_j}\in \{E_{e_1},\ldots,E_{e_k}\}$. Note that $\pi_{e_j}(G')$ is the tope graph of $\mathcal{S}(\covectors(G)/e)\backslash e_j$ and therefore it is in $\topegraphs$.


By Theorem~\ref{thm:dictionaryOMC}, the covectors of the COM corresponding to $\pi_{e_j}(G')$ are precisely its antipodal subgraphs. Since $\pi_{e_j}(A')$ is antipodal, it corresponds to a covector. But then this covector is in $\mathcal{S}((\covectors(G)\backslash e_j)/e)$, i.e.~there is an antipodal graph in $\pi_{e_j}(G)$ whose zone graph is $\pi_{e_j}(A')$. By definition this must be $\pi_{e_j}(A)$, proving that $\pi_{e_j}(A)$ is antipodal for every $E_{e_j}\in \{E_{e_1},\ldots,E_{e_k}\}$.


Let $v\in A$ be a vertex without an antipode in $A$. Without loss of generality, $v\in E_e^+$. By Lemma~\ref{lem:allcontractionsantipodal}, $A$ either has a pc-minor in $\Q$ and we are done, or there is a $Q_k^-$ or $Q_k^{--}$ in $A$ crossed by precisely  $E_{e_1},\ldots,E_{e_k}$ and its missing vertex is the missing antipode of $v$ in $A$. Then this convex subgraph is precisely $E_e^- \cap A$.

First assume that $E_e^-\cap A$ is isomorphic to $Q_k^-$. Thus, $v$ has a neighbor in $E_e^-$, say $u$. If $u$ has no antipode in $A$ we deduce as above that $E_e^+$ is isomorphic to $Q_k^-$ or $Q_k^{--}$. Since $v\in A$ the halfspace $E_e^+$ must be isomorphic to $Q_k^-$ and thus $A \cong Q_k^- \square K_2$. But then the zone graph of $A$ corresponding to $E_e$ is not antipodal and $A'$ is not antipodal. A contradiction.
 
Finally, assume that $E_e^-$ is isomorphic to $Q_k^{--}$. Then there are $k$ vertices in $E_e^-$ at distance 2 from $v$ and pairwise also at distance 2 but there is no vertex adjacent to all of them. By Lemma~\ref{lem:K_n}, $A$ either has a pc-minor in $\Q$ and we are done, or $A$ is isomorphic to $Q_{k+1}^-$ or $Q_{k+1}^{--}$. In the first case, none of the zone graphs of $Q_{k+1}^-$ is antipodal, while in the second case $A$ is antipodal. A contradiction. This finishes the proof that $\covectors(G)/e = \covectors(G')$.

\end{proof}

Now we can assume that $G'$ is a well-embedded partial cube, but is not in $\topegraphs$ since $\covectors(G')=\mathcal{S}(\covectors(G)/e)$, and $\covectors(G)/e$ is not a COM. By minimality of $G$, $G'=\zeta_e(G)$ has a pc-minor $H'\in\Q$, i.e., $H'=\rho_X(\pi_A(\zeta_e(G)))$ for some $\Theta$-classes $A$ and an oriented set $X$ of $\Theta$-classes of $G'$. By Lemma~\ref{lem:hyperplanescommute} we have $H'=\zeta_e(\rho_{X'}(\pi_{A'}(G)))$ for $\Theta$-classes $A'$ and an oriented set $X'$ of $\Theta$-classes of $G$. Let $H$ be the graph $\rho_{X'}(\pi_{A'}(G))$. Lemma~\ref{lem:hyperplanesinQ} gives that $H$ has a pc-minor in $\Q$, contradicting that $G\in\mathcal{F}(\Q)$. This concludes the proof of Theorem~\ref{thm:topegraphs}. 

%
%
%
%
\end{proof}

%
%
%
%
%
%
%
%
\section{Antipodally gated partial cubes are pc-minor closed}\label{sec:antipodalminorclosed}
The main result of the present section is that if in a partial cube all antipodal subgraphs are gated, then the same holds for all its minors (Theorem~\ref{thm:AGminors}). 
Recall that the class of these antipodally gated partial cubes is denoted by $\AG$. Since by Lemma~\ref{lem:minfamily} none of the graphs in $\Q$ is in $\AG$, minor-closedness of $\AG$ implies that antipodally gated partial cubes exclude minors from $\Q$, i.e., $\AG\subseteq \mathcal{F}(\Q)$. This section is proof wise the hardest one of the paper. It heavily builds on interactions of antipodality and gatedness with respect to pc-minors, expansions, and zone-graphs established in Section~\ref{sec:pcminors}. We want to prove:

\begin{thm}\label{thm:AGminors}
If $G$ is antipodally gated, then so are all pc-minors of $G$.
\end{thm}

For this we will show two auxiliary statements. The first one is:

\begin{lem}\label{lem:expansion}
Let $G$ be an antipodal graph from $\AG$ such that all its pc-minors are in $\AG$ as well and $G'$ an expansion of $G$. Then one of the following occurs:
\begin{enumerate}
\item $G'$ is antipodal,
\item $G'$ is a peripheral expansion of $G$,
\item $G'$ is not in $\AG$.
\end{enumerate}
\end{lem}

The second one shows how to use the first one in order to prove Theorem~\ref{thm:AGminors}.

\begin{lem}\label{lem:exp->AGminors}
 If Lemma~\ref{lem:expansion} holds for all pairs $G$, $G'$ where $G'$ is on less than $n$ vertices, then Theorem~\ref{thm:AGminors} holds for all the partial cubes on at most $n$ vertices.
\end{lem}

\begin{proof}[Proof of Lemma~\ref{lem:exp->AGminors}.]
 Suppose that Theorem~\ref{thm:AGminors} does not hold and let $G$ be a minimal counterexample, while Lemma~\ref{lem:expansion} holds for all the expansions of size less than the size of $G$. First, observe that $\AG$ is closed under restrictions, since restrictions cannot create new antipodal subgraphs and gated subgraphs remain gated. So, let $\pi_e(G)\notin \AG$ be a contraction of $G$ that is not in $\AG$. Let $A$ be a smallest antipodal subgraph of $\pi_e(G)$, that is not gated in $\pi_e(G)$. In particular $A$ is a proper subgraph. By the minimality in the choice of $A$, itself is in $\AG$. Now, by the minimality in the choice of $G$, all pc-minors of $A$ are also in $\AG$. Let $A'$ denote the expansion of $A$ with respect to $e$, that appears as a proper subgraph of $G$. If $E_e$ does not cross $A'$, then $A'\cong A$ is  antipodal subgraph and is non-gated, since otherwise $A=\pi_e(A')$ would be gated as well by Lemma~\ref{contraction_gated}. This contradicts $G\in \AG$. If $E_e$ crosses $A'$, we can apply Lemma~\ref{lem:expansion} to $A$, since $A'$ has less vertices than $G$. We get that either $A'$ is antipodal, $A'$ is a peripheral expansion of $A$, or $A'$ is not in $\AG$. The latter cannot be since $G$ in $\AG$. In the former two cases, either $A'$ is antipodal or has $A$ as a subgraph. In both cases, we have an antipodal subgraph that is contracted to $A$ in $\pi_e(G)$. By Lemma~\ref{contraction_gated}, the antipodal subgraph in $G$ is non-gated contradicting $G\in \AG$.
\end{proof}

\begin{proof}[Proof of Lemma~\ref{lem:expansion}.]
Suppose that the lemma is false. Let $G$, $G'$ be a minimal counterexample, i.e.~$G$ is an antipodal graph from $\AG$ such that all its pc-minors are in $\AG$. Furthermore, $G'$ is an expansion of $G$ that is not antipodal, not peripheral, but in \AG, with minimal number of vertices possible. 
Let $E_c$ be the $\Theta$-class such that $\pi_c(G')=G$.


\begin{claim}\label{claim:AG}
 Any pc-minor of $G'$ is in $\AG$.
\end{claim}
 \begin{proof}
Let $n$ be the number of vertices in $G'$. Since $G'$ is a minimal counterexample to Lemma~\ref{lem:expansion}, the lemma holds for all the graphs on less than $n$ vertices. Then by Lemma~\ref{lem:exp->AGminors}, Theorem~\ref{thm:AGminors} holds for all graphs on at most $n$ vertices. In particular it holds for $G'$, thus all its pc-minors are in $\AG$.
\end{proof}

We will call a contraction of a partial cube \emph{antipodal} if the contracted graph is antipodal. The following claim is immediate since a contraction of an antipodal graph is antipodal and contractions commute.

\begin{claim}\label{claim:antpodaltominors}
If $\pi_e(H)$ is an antipodal contraction of $H$, then  $\pi_e(\pi_f(H))$ is an antipodal contraction of $\pi_f(H)$ for all $f\in\mathcal{E}$.
\end{claim}

Every contraction $\pi_e(G')$ of $G'$, $E_e\neq E_c$, makes $E_c$ peripheral or it is antipodal, since otherwise the contraction $\pi_e(G')$ together with the contraction $\pi_c(\pi_e(G'))$ would yield a smaller counterexample to the lemma. We can divide the $\Theta$-classes of $G'$ into two sets: call the index set of the $\Theta$-classes of the antipodal contractions $\mathcal{A}$, and the index set of the remaining $\Theta$-classes $\mathcal{B}$. By the above, a  contraction $\pi_e(G')$ of $G'$, for every $e\in \mathcal{B}$, makes $E_c$ peripheral in $\pi_e(G')$. Note that $c\in \mathcal{A}$, i.e.~in particular $\mathcal{A}$ is non-empty. Also $\mathcal{B}$ is non-empty, because otherwise, every contraction of $G'$ is antipodal, thus by Lemma~\ref{lem:allcontractionsantipodal}, $G'$ is isomorphic to $Q_n^{-}$ or $Q_n^{--}$. The latter cannot be since $G'$ is not antipodal. The former is impossible since then $G=\pi_c(G')\cong Q_{n-1}$, and $G'$ is a peripheral expansion. 

Furthermore, note that peripherality of a $\Theta$-class is preserved under contraction.

 \begin{claim}\label{claim:symmetry}
 For every $e\in \mathcal{B}$ and every $f\in \mathcal{A}$, the $\Theta$-class $E_f$ is peripheral in $\pi_e(G')$.
 \end{claim}
 \begin{proof}
Assume that this is not the case. Then $\pi_e(G')$ is not antipodal by definition, while $\pi_f(\pi_e(G'))$ is antipodal by Claim~\ref{claim:antpodaltominors}. Moreover, $E_f$ is not peripheral in $\pi_e(G')$ by assumption, while all pc-minors $\pi_e(G')$  are in $\AG$ by Claim~\ref{claim:AG}. Thus $\pi_e(G')$ and $\pi_f(\pi_e(G'))$ are a smaller counterexample to the lemma, a contradiction.
 \end{proof}
 
Now, consider the halfspaces $E_c^+$ and $E_c^-$ in $G'$. Contracting any $\Theta$-class $e$ from $\mathcal{B}$ makes $E_c$ peripheral in $\pi_e(G')$, which implies that either $\pi_e(E_c^+)$ or $\pi_e(E_c^-)$ is a peripheral halfspace in $\pi_e(G')$. For this reason denote by $\mathcal{B}^+$ those $e\in\mathcal{B}$ such that $\pi_e(E_c^+)$ is peripheral and let $\mathcal{B}^-=\mathcal{B}\setminus\mathcal{B}^+$, i.e., $\pi_e(E_c^-)$ is peripheral for $e\in \mathcal{B}^-$.

\begin{claim}\label{claim:beta+-}
Let $e \in \mathcal{B}^+$ end $f \in \mathcal{B}^-$. Then $\pi_f(\pi_e(G'))$ is antipodal. Moreover, $E_f$ is peripheral in $\pi_e(G')$.
\end{claim}

\begin{proof}
Since peripherality is closed under contraction, both $\pi_f(\pi_e(E_c^+))$ and $\pi_f(\pi_e(E_c^-))$ are peripheral in $\pi_f(\pi_e(G'))$. Thus $\pi_f(\pi_e(G'))\cong K_2 \square A$ for some graph $A$. On one hand contracting $E_c$ in $\pi_f(\pi_e(G'))$ gives an antipodal graph by the choice of $E_c$, on the other hand it is isomorphic to $A$. Thus $A$ is an antipodal graph. Then also $\pi_f(\pi_e(G'))\cong K_2 \square A$ is antipodal.

Now consider the pair of graphs $\pi_f(\pi_e(G'))$ and $\pi_e(G')$. The first is antipodal by the above, while the second is its expansion. Since both are pc-minors of $G'$,  $\pi_e(G')$ is in $\AG$ and $\pi_f(\pi_e(G'))$ and all its pc-minors are in $\AG$ as well, by Claim~\ref{claim:AG}. Furthermore, $\pi_e(G')$ is not antipodal since $e\in \mathcal{B}$. By the minimality of $G'$ the expansion $\pi_e(G')$ must be peripheral, proving that $E_f$ is peripheral in $\pi_e(G')$.
\end{proof}

Let $G''$ be obtained from $G'$ by a maximal chain of contractions of $\Theta$-classes $\mathcal{C}=\{ {e_1},\ldots,{e_p}\}$, such that $G''$ is non-antipodal. Note that $\mathcal{C} \subseteq \mathcal{B}$ since every contraction from a $\Theta$-class in $\mathcal{A}$ makes the graph antipodal.
Moreover, by Claim~\ref{claim:beta+-} we have $\mathcal{C} \subseteq \mathcal{B}^+$ or $\mathcal{C} \subseteq \mathcal{B}^-$. Without loss of generality assume that $\mathcal{C} \subseteq \mathcal{B}^+$.

\begin{claim}\label{claim:hypercubebeta-}
The contraction $\pi_{e_1}(G')$ contains a hypercube that is crossed exactly by $\mathcal{B}^+\setminus\{e_1\}$.
\end{claim}

\begin{proof}
 By assumption $E_c^+$ is not peripheral in $G'$, hence there exists a vertex $v \in E_c^+$ not incident with any edge of $E_c^+$. For every $e \in \mathcal{B}^+$, $E_c^+$ is peripheral in $\pi_e(G')$, thus there is an edge in $E_e$ connecting $v$ and an edge in $E_c$. Thus for every $e \in \mathcal{B}^+$, there is a path from $v$ to $E_c^-$ first crossing an edge in $E_e$ and then an edge in $E_c$. Vertex $v$ together with the end-vertices of these paths form a collection of vertices pairwise at distance 2. They have no common neighbor, since this neighbor would have to be in $E_c^-$, but $v$ has no neighbor in $E_c^-$. Since all pc-minors of $G'$ are in $\AG$, Lemma~\ref{lem:K_n} implies that the convex hull $C$ of these vertices is isomorphic to $Q_n^-$ or $Q_n^{--}$ for some $n>0$. This convex subset is crossed exactly by all the $\Theta$-classes of $\mathcal{B}^+$ and $E_c$.

Contracting to $\pi_{e_1}(G')$, $C$ is is contracted into a hypercube crossed by all $\mathcal{B}^+\setminus\{e_1\}$ and $E_c$. Then the lemma holds.
\end{proof}

The graph $G''$ and all its pc-minors are in $\AG$ by Claim~\ref{claim:AG} and all its contractions are antipodal. By Lemma~\ref{lem:allcontractionsantipodal}, we have $G''\cong Q_n^-$. Now, we will consider the sequence of expansions of $G''$ leading back the first contraction $\pi_{e_1}(G')$. Note that $G''$ is crossed by precisely those $\Theta$-classes that are not in $\mathcal{C}$. These are the $\Theta$-classes $\mathcal{A} \cup \mathcal{B}^-$, and the $\Theta$-classes $\mathcal{B}^+\backslash \mathcal{C}$.  

By Claims~\ref{claim:symmetry} and~\ref{claim:beta+-}, every $\Theta$-class in $\mathcal{A}\cup \mathcal{B}^-$ is peripheral in $\pi_{e_1}(G')$ thus also in every contraction of it. For each $e \in \mathcal{A}\cup \mathcal{B}^- \backslash \{c\}$, without loss of generality, say that $E_e^+$ is the peripheral halfspace of it. For $E_c$ the halfspace $E_c^+$ is peripheral, since ${e_1}\in \mathcal{B}^+$, so we can assume that $E_e^+$ is peripheral for each $e \in \mathcal{A}\cup \mathcal{B}^-$.
Since $G''\cong Q_n^-$ is crossed by each $E_e$ and $E_e^-$ is non-peripheral in it, $E_e^-$ is also non-peripheral in every expansion in the sequence.

Let $|\mathcal{A} \cup \mathcal{B}^-| = k$ and $|\mathcal{B}^+\backslash \mathcal{C}|=\ell$. Then $G''\cong Q_n^-$ with $n=k+\ell$, can be seen as a collection of $2^k$ disjoint subgraphs spanned on the edges of $\mathcal{B}^+\backslash \mathcal{C}$ each isomorphic to $Q_\ell$, except for one isomorphic to $Q_{\ell}^-$, and all connected in a hypercube manner by edges of $\mathcal{A} \cup \mathcal{B}^-$. The subgraph isomorphic to $Q_{\ell}^-$ is precisely the subgraph $\bigcap_{e\in \mathcal{A}\cup \mathcal{B}^-} E_e^+$. 

In other words $\pi_{\mathcal{B}^{+}\backslash\mathcal{C}}(G'')\cong Q_k$ and for $X\in\{+, -\}^{\mathcal{A} \cup \mathcal{B}^-}$ we have $\rho_X(G'')\cong Q_{\ell}$ unless if $X=(+,\ldots, +)$ in which case $\rho_X(G'')\cong Q^-_{\ell}$. Next we prove that this structure is preserved when expanding back towards $\pi_{e_1}(G')$.

\begin{claim}\label{claim:structpiG'}
Let $\bar{\ell}=\ell+p-1$, where $p=|\mathcal{C}|$. We have $\pi_{\mathcal{B}^{+}\backslash\mathcal{C}}(\pi_{e_1}(G'))\cong Q_k$ and for $X\in\{+, -\}^{\mathcal{A} \cup \mathcal{B}^-}$ we have $\rho_X(\pi_{e_1}(G'))\cong Q_{\bar{\ell}}$ if and only if $X\neq(+,\ldots, +)$.

%
%
%
\end{claim}


\begin{proof}
We will prove the claim by induction on the number $p$, where $G_j$ is the $j$th expansion, starting from $G''=G_0$ and ending in $\pi_{e_1}(G')=G_{p-1}$. By the paragraph before the claim, the claim holds for $G''$.

By induction assumption let the claim hold for $G_j$, $j\geq 0$, and let $G_{j+1}$ be its expansion, in the sequence of the expansions leading to $\pi_{e_1}(G')$. Let $f\in \mathcal{C}$ be such that $\pi_f(G_{j+1})=G_j$ and $H_1,H_2$ be the subgraphs of $G_j$ we expand along.

First, we will prove that each copy of $Q_{\ell+j}$ in $G_j$ not crossed by the $\Theta$-classes $\mathcal{A}\cup \mathcal{B}^-$ is contained in $H_1\cap H_2$.
Denote $\mathcal{A} \cup \mathcal{B}^-=\{{f_1},\ldots,{f_s}\}$. Consider an edge $uv\in E_{f_i}$ in $G_j$ such that $u\in H_1\cap H_2$. If $u\in E_{f_i}^+$ and $v\in E_{f_i}^-$, then $E_{f_i}^+$ is peripheral in the expanded graph $G_{j+1}$. Thus, $v\in H_1\cap H_2$. 

Now assume that $u\in E_{f_i}^-$ and $v\in E_{f_i}^+$. We will prove that if additionally $v\notin\bigcap_{e\in \mathcal{A}\cup \mathcal{B}^-} E_e^+$, then we can also conclude $v\in H_1\cap H_2$. So let $u,v \in E_{f_{i'}}^-$ for some other ${f_{i'}}\in \mathcal{A} \cup \mathcal{B}^-\setminus\{{f_{i}}\}$. For the sake of contradiction assume that $v\notin H_1\cap H_2$. Without loss of generality $v\in H_1$. If $v$ has a neighbor in $E_{f_{i'}}^+$ -- say $v'$, then by the above arguments, if $v'\in H_1\cap H_2$ then also $v\in H_1\cap H_2$. Thus, if $v'$ exists, then it is in $H_1 \backslash H_2$.

Consider the contractions $\pi_{f_{i'}}(G_{j+1})$ and $\pi_{f_{i'}}(G_{j})$. Since ${f_{i'}}\in \mathcal{A} \cup \mathcal{B}^-$ both  graphs are antipodal. By Lemma~\ref{lem:G1G2}, the expansion from $\pi_{f_{i'}}(G_{j})$ to $\pi_{f_{i'}}(G_{j+1})$ corresponds to sets $ \pi_{f_{i'}}(H_1)$ and $\pi_{f_{i'}}(H_2)$. Since $v,v'\in H_1\backslash H_2$, their image $\pi_{f_{i'}}(v)=\pi_{f_{i'}}(v')$ is in $\pi_{f_{i'}}(H_1)\backslash \pi_{f_{i'}}(H_2)$. Since $u\in H_1\cap H_2$,  its image is in $\pi_{f_{i'}}(H_1)\cap \pi_{f_{i'}}(H_2)$. Let $-\pi_{f_{i'}}(v),-\pi_{f_{i'}}(u)$ be the antipodes of $\pi_{f_{i'}}(v),\pi_{f_{i'}}(u)$ in $\pi_{f_{i'}}(G_{j})$, respectively. Since $\pi_{f_{i'}}(G_{j+1})$ is an antipodal expansion of $\pi_{f_{i'}}(G_{j})$ by Lemma~\ref{lem:antipodalexpansion} we have,
 $-\pi_{f_{i'}}(v)\in \pi_{f_{i'}}(H_2)\backslash \pi_{f_{i'}}(H_1)$ and $-\pi_{f_{i'}}(u) \in \pi_{f_{i'}}(H_1)\cap \pi_{f_{i'}}(H_2)$. 
 
 Since $v\in E_{f_i}^+$ also $\pi_{f_{i'}}(v)\in E_{f_i}^+$, thus $-\pi_{f_{i'}}(v)\in E_{f_i}^-$. Similarly, since $u\in E_{f_i}^-$ also $\pi_{f_{i'}}(u)\in E_{f_i}^-$, thus $-\pi_{f_{i'}}(u) \in E_{f_i}^+$.
 
 Consider the contraction from $G_j$ to $\pi_{f_{i'}}(G_{j})$. By induction hypothesis, the structure of $G_j$ is such that $E_{f_i}^-$ is isomorphic to a hypercube crossed by $E_{f_{i'}}$ thus two vertices $x,z$  of $G_j$ are contracted to $-\pi_{f_{i'}}(v)$. Since $-\pi_{f_{i'}}(v)\in \pi_{f_{i'}}(H_2)\backslash \pi_{f_{i'}}(H_1),$ we have $x,z\in H_2\backslash H_1$. 
 
 On the other hand, at least one vertex $w\in H_1\cap H_2$ of $G_j$ is contracted to $-\pi_{f_{i'}}(u)$. 
It also holds that the vertices $x,z\in E_{f_i}^-$ and $w\in E_{f_i}^+$. Since $uv$ is an edge, $-\pi_{f_{i'}}(u)-\pi_{f_{i'}}(v)$ is an edge. Hence, $xw$ or $zw$ has to be an edge - say without loss of generality $zw$ is an edge. Then we have a pair of adjacent vertices $w,z$ such that $w\in H_1 \cap H_2\cap  E_{f_{i}}^+$ and $z\in H_1 \backslash  H_2\cap E_{f_{i}}^-$. We have already shown that this is impossible. Thus, $u\in H_1 \cap H_2$ implies $v \in H_1 \cap H_2$.
 
Now consider the copies of $Q_{\ell+j}$ in $G_{j}$. By Claim~\ref{claim:hypercubebeta-}, $\pi_{e_1}(G')$ contains a hypercube crossed by exactly all the $\Theta$-classes of $\mathcal{B}^+\backslash \{e_1\}$. The latter implies that $G_{j+1}$ contains a hypercube crossed by precisely the $\Theta$-classes crossing copies $Q_{\ell+j}$ in $G_{j}$ plus the $\Theta$-class obtained while expanding to $G_{j+1}$. In particular this implies that one of the copies of $Q_{\ell+j}$ in $G_{j}$ gets completely expanded, i.e.~that copy is completely in $H_1\cap H_2$. Since none of the copies of $Q_{\ell+j}$ in $G_{j}$ is in $\bigcap_{e\in \mathcal{A}\cup \mathcal{B}^-} E_e^+$, by the above paragraph this property propagates to all the copies of $Q_{\ell+j}$, i.e.~all of them are completely in $H_1\cap H_2$. Since the non-cube $\bigcap_{e\in \mathcal{A}\cup \mathcal{B}^-} E_e^+$ cannot expand to a cube, $G_{j+1}$ is as stated in the claim.
\end{proof}

The following claim will suffice to finish the proof of the lemma. 

\begin{claim}
If the expansion $G'$ of $\pi_{e_1}(G')$ is such that $\pi_{c}(G')$ is antipodal and $E_c$ is not peripheral in $G'$, then $G'$ is not in $\AG$.
\end{claim}
\begin{proof}

Consider the expansion along sets $G_1$ and $G_2$. Since $e_1\in\mathcal{B}^+$, the halfspace $E_c^+$ is peripheral in $\pi_{e_1}(G')$. Moreover, by Claim~\ref{claim:structpiG'} we have $E_c^-\cong Q_{k+\bar{\ell}-1}$ and $E_c^+ \not \cong Q_{k+\bar{\ell}-1}$. On the other hand, since $E_c$ is not peripheral in $G'$, there is an edge $zz'\in \pi_{e_1}(G')\cap E_c$ such that $z\in (G_1 \backslash G_2)\cap E_c^-$ and $z'\in G_1\cap G_2\cap E_c^+$.

Note that $\pi_{c}(G')$ is an antipodal expansion of $\pi_{e_1}(\pi_{c}(G')) \cong Q_{k+\bar{\ell}-1}$. By Lemma~\ref{lem:G1G2}, it is an expansion along sets $\pi_c(G_1)$ and $\pi_c(G_2)$. In the following we will use Lemma~\ref{lem:G1G2} to get properties of $G_1, G_2$ from the expansion $\pi_c(G_1),\pi_c(G_2)$.


We consider two cases, either $\pi_{c}(G')$ is a full expansion of $\pi_{e_1}(\pi_{c}(G')) \cong Q_{k+\bar{\ell}-1}$ or not. In the first case, every vertex of  $ \pi_{e_1}(\pi_{c}(G'))$ is in $\pi_c(G_1)\cap \pi_c(G_2)$, thus for any edge $E_{c}$ in $ \pi_{e_1}(G')$  we have that one of its endpoints is in $G_1\cap G_2$ and any vertex not incident to $E_{c}$ is in $G_1\cap G_2$, as well. 
Since $E_{c}^+ \not \cong Q_{k+\bar{\ell}-1}$ there is a missing vertex in it.

Let $v$ be a vertex in $(G_1\backslash G_2)\cap E_{c}^-$ that is as close as possible to a missing vertex. We have proved above that $v$ exists (one candidate would be $z$). Since the expansion is full, by the above paragraph, $v$ is not adjacent to the missing vertex. Then by the above paragraph its neighbor in $E_{c}^+$ and by the choice of $v$ all its neighbors in $E_{c}^-$ closer to the missing vertex are in $G_1\cap G_2$. In the expansion $G'$, the expansions of these vertices together with $v$ give rise to vertices pairwise at distance 2. Since $v$ is not expanded they have no common neighbor. By Claim~\ref{claim:AG} $G'$ and all its pc-minors are in $\AG$, thus by Lemma~\ref{lem:K_n}, their convex hull $C$ is a cube minus a vertex or a cube minus two antipodes. Moreover, by the choice of $v$ and since $E_c^-\cong Q_{k+\bar{\ell}-1}$ in $\pi_{e_1}(G_1)$, the hull $C$ also has the missing vertex. However, now also its expansion is missing, thus two adjacent vertices are missing from $C$. Thus, $C$ is not isomorphic to a cube minus a vertex or a cube minus two antipodes.

Now consider that $\pi_{c}(G')$ is not a full expansion of $\pi_{e_1}(\pi_{c}(G')) \cong Q_{k+\bar{\ell}-1}$. Then there exists a vertex $v$ of $\pi_{e_1}(\pi_{c}(G'))$ with $v\in\pi_c(G_2)\backslash \pi_c(G_1)$. Then there is a vertex $w$ of $\pi_{e_1}(G')$ in the preimage of $v$ such that $w\in(G_2\backslash G_1)\cap E_c^-$. Let again $zz' \in \pi_{e_1}(G')$ be an edge as in the first paragraph, i.e.~$zz'\in E_c$ such that $z\in (G_1\backslash G_2)\cap E_c^-$ and $z'\in G_1\cap G_2 \cap E_c^+$. Then let $y\in E_c^-$ be the closest vertex to $z$ such that $y\in G_2\backslash G_1$ (one candidate would be $w$). Now, let $x$ be a vertex in the interval from $z$ to $y$ such that $x\in G_1\backslash G_2$ and $x$ is as close as possible to $y$. Then all the vertices in the interval from $x$ to $y$ are in $G_1\cap G_2$ except for $x,y$. Then since $E_c^-\cong Q_{k+\bar{\ell}-1}$ the interval from $x$ to $y$ is isomorphic to a hypercube, i.e.~antipodal. Then, also its expansion is an antipodal graph, but one preimage of $z'$ has no gate in it by Lemma~\ref{lem:nongatedexpansion}.
\end{proof}

By the above claim the sequence of contractions cannot exist which gives a contradiction. Thus there is no $G',G$ and the lemma holds.
%
\end{proof}

\begin{proof}[Proof of Theorem~\ref{thm:AGminors}.]
 Since Lemma~\ref{lem:expansion} holds, by Lemma~\ref{lem:exp->AGminors} we obtain the theorem.
\end{proof}

\section{Further characterizations and recognition of tope graphs}\label{sec:corollaries}
In this section we will describe the important implications of Theorem~\ref{thm:main}. It can be read without having gone through the technicalities of the previous sections. In particular, we give polynomial time recognition algorithm for tope graphs of COMs, specialize our results to tope graphs of OMs, AOMs, and LOPs (of bounded rank), and prove a conjecture of~\cite{Che-16}.
But first of all, Theorem~\ref{thm:main} can be used to obtain a characterization of $\topegraphs$ in terms of zone graphs that is a generalization of a result of Handa~\cite{ha-90}.

\begin{cor}\label{cor:Handa}
 A graph $G$ is the tope graph of a COM, i.e. $G\in\topegraphs$, if and only if $G$ is a partial cube such that all iterated zone graphs are well-embedded partial cubes.
\end{cor}
\begin{proof}
 If $G\in\topegraphs$, then by Lemma~\ref{lem:COMhyperplane} all its zone graphs are well-embedded partial cubes and in $\topegraphs$. Hence, the argument can be iterated to prove that $\zeta_{e_1,\ldots,e_k}(G)$ is a partial cube for any sequence of hyperspaces.
 
 If $G\notin\topegraphs$, then by Theorem~\ref{thm:main} there is an $H\in\Q$ such that $H=\pi_A(\rho_X(G))$ for some $\Theta$-classes of $G$. By Lemma~\ref{lem:hyperplanesofQ} we can thus find a sequence $e_1,\ldots,e_k$ such that $\zeta_{e_1,\ldots,e_k}(\pi_A(\rho_X(G)))$ is not a partial cube. If $\zeta_{e_i,\ldots,e_k}(G)$ was a well-embedded partial cube for all $1\leq i\leq k$, then by Lemma~\ref{lem:hyperplanescommute} we could find sets of $\Theta$-classes $A'$ and $X'$ in it such that $\zeta_{e_1,\ldots,e_k}(\pi_A(\rho_X(G)))=\pi_{A'}(\rho_{X'}(\zeta_{e_1,\ldots,e_k}(G)))$. But then the latter would be a partial cube, contradicting the choice of $e_1,\ldots,e_k$. Therefore there is an iterated zone graph of $G$ that is not a partial cube.
\end{proof}

Theorem~\ref{thm:main} and Corollary~\ref{cor:Handa} specialize to other systems of sign-vectors.
Using Proposition~\ref{prop:OMandCOM} we immediately get the following corollary. Recall that for a set ${X}$ of partial cubes we denote by $\mathcal{F}({X})$ the class of partial cubes that do not have any graph from $X$ as a partial cube minor.
\begin{cor}\label{cor:OM}
For a graph $G$ the following conditions are equivalent:
 \begin{itemize}
  \item[(i)] $G$ is the tope graph of an OM, i.e., $G\in\topegraphsOM$,
  \item[(ii)] $G$ is an antipodal partial cube and all its antipodal subgraphs are gated,
  \item[(iii)] $G$ is in $\mathcal{F}(\Q)$ and antipodal,
  \item[(iv)] $G$ is an antipodal partial cube and all its iterated zone graphs are well-embedded partial cubes.
 \end{itemize}
\end{cor}

Note that the equivalence (i)$\Leftrightarrow$(ii) corresponds to a characterization of tope sets of
OMs due to da Silva~\cite{daS-95} and (i)$\Leftrightarrow$(vi) corresponds to a characterization of tope sets of  Handa~\cite{Han-93}.


Let us call an affine subgraph $G'$ of an affine partial cube $G$ \emph{conformal} if for all $v\in G'$ we have $-_{G'}v\in G'\Leftrightarrow -_Gv\in G$. We give an intrinsic characterization of $\topegraphsAOM$:
\begin{cor}\label{cor:AOM}
For a graph $G$ the following conditions are equivalent:
 \begin{itemize}
  \item[(i)] $G$ is the tope graph of an AOM, i.e., $G\in\topegraphsAOM$,
  \item[(ii)] $G$ is an affine partial cube and all its antipodal and conformal subgraphs are gated,
  \item[(iii)] $G$ is in $\mathcal{F}(\Q)$ (or $G\in\topegraphs$), affine, and all its conformal subgraphs are gated.
 \end{itemize}
\end{cor}
\begin{proof}
 By Proposition~\ref{prop:affOMandOM}, $G\in\topegraphsAOM$ if and only if $G$ is an affine partial cube such that the antipodal $\widetilde{G}$ containing $G$ as a halfspace is in $\topegraphsOM$. By Corollary~\ref{cor:OM} this is equivalent to all antipodal subgraphs of $\widetilde{G}$ being gated. It is easy to see, that the antipodal subgraphs of $\widetilde{G}$ correspond to the antipodal and the conformal subgraphs of $G$ and they are gated if and only if the corresponding antipodal and conformal subgraphs of $G$ are. This proves (i)$\Leftrightarrow$(ii).
 
\end{proof}

%

For lopsided sets the (i)$\Leftrightarrow$(ii) part of the corollary corresponds to a characterization due to Lawrence~\cite{Law-83}. For the complete statement denote $\mathcal{Q}^{--}:=\{{Q}_n^{--}\mid n\geq 3\}$.
\begin{cor}\label{cor:LOP}
For a graph $G$ the following conditions are equivalent:
 \begin{itemize}
  \item[(i)] $G$ is the tope graph of a LOP, i.e., $G\in\topegraphsLOP$,
  \item[(ii)] $G$ is a partial cube and all its antipodal subgraphs are hypercubes,
  \item[(iii)] $G$ is in $\mathcal{F}(\mathcal{Q}^{--})$.
 \end{itemize}
\end{cor}
\begin{proof}
 By Proposition~\ref{prop:LOPandCOM} and Theorem~\ref{thm:main} $G\in\topegraphsLOP$ is equivalent to the property that $G$ has all antipodal subgraphs gated and isomorphic to hypercubes. But hypercubes are always gated. This gives (i)$\Leftrightarrow$(ii).
 
The implication (ii)$\Rightarrow$(iii) follows from the fact that $\topegraphsLOP$ is pc-minor closed while the graphs from $\mathcal{Q}^{--}$ are (pc-minor minimal) non-hypercube antipodal graphs.
 
To prove (iii)$\Rightarrow$(ii), let $G$ be a pc-minor-minimal  partial cube with an antipodal subgraph not isomorphic to a hypercube. Then $G$ is antipodal and not isomorphic to a hypercube. Let $G$ be embedded in a $Q_n$ for some minimal $n\in \mathbb{N}$. Now, $G$ can be seen as $Q_n$ minus a set of antipodal pairs of vertices. Take a pair $u,v$ of vertices in $Q_n$ not in $G$ at a minimal distance in $Q_n$. If $d(u,v)=1$, contracting the $\Theta$-class of $uv$ contract $G$ to an antipodal subgraph not isomorphic to a hypercube, contradicting that $G$ is pc-minor-minimal. The interval $\conv(u,v)$ in $Q_n$ intersected with $G$ is a convex subgraph of $G$ isomorphic to a hypercube minus a vertex. Note that $d(u,v)>2$, since otherwise the subgraph is not connected. Since $G$ is pc-minor-minimal, the subgraph is the whole $G$ and thus $G$ is in $\mathcal{Q}^{--}$.
\end{proof}

Recall that by Lemma \ref{lem:rankchar} the rank of a COM is the dimension of a maximal hypercube to which its tope graph can be contracted. 
Considering COMs of bounded rank, we can reduce the set of excluded pc-minors to a finite list. For any $r\geq 3$ define the following finite sets $$\Q_r:=\{Q^{-*}_n,Q^{--}_{n}(m), Q^{--}_{r+2}(r+2), Q_{r+1}\mid 4\leq n\leq r+1; 1\leq m\leq n\}\subset \Q \cup \{Q_{r+1}\}, $$ 
and
$$\mathcal{Q}^{--}_r:=\{Q^{--}_{n}, Q_{r+1} \mid 3\leq n\leq r+1\}\subset \mathcal{Q}^{--} \cup \{Q_{r+1}\}.$$

\begin{cor}\label{cor:boundedrank}
 For a graph $G$ and an integer $r\geq 3$ we have:
 \begin{itemize}
  \item $G\in\topegraphs$ of rank at most $r$ $\Leftrightarrow$ $G\in\mathcal{F}(\Q_r)$.
  \item $G\in\topegraphsOM$ of rank at most $r$ $\Leftrightarrow$ $G\in\mathcal{F}(\Q_r)$ and $G$ is antipodal.
  \item $G\in\topegraphsAOM$ of rank at most $r$ $\Leftrightarrow$ $G\in\mathcal{F}(\Q_r)$, $G$ is affine and all its conformal subgraphs
  are gated.
  \item $G\in\topegraphsLOP$ of rank at most $r$ $\Leftrightarrow$ $G\in\mathcal{F}(\mathcal{Q}^{--}_r)$.
 \end{itemize}
\end{cor}
\begin{proof}
 The respective proofs follow immediately from Theorem~\ref{thm:main} and Corollaries~\ref{cor:OM}, ~\ref{cor:AOM}, and~\ref{cor:LOP} together with the observation that the largest hypercube that is a contraction minor of $Q^{-*}_n$ and $Q^{--}_n(m)$ is of dimension $n-1$ for $0\leq m\leq n-1$ and of $Q^{--}_{n}(n)\in\Q$ is of dimension $n-2$. Now, the graphical interpretation of the rank as given in Lemma~\ref{lem:rankchar} gives the result.
\end{proof}

Note that using Proposition~\ref{prop:polyminor} can easily be seen to yield a polynomial time recognition algorithm for the recognition of the bounded rank classes above. However, Theorem~\ref{thm:main} also yields polynomial time recognition algorithms for the unrestricted classes.

\begin{cor}\label{cor:recognition}
 The recognition of graphs from any of the classes $\topegraphs, \topegraphsAOM, \topegraphsOM, \topegraphsLOP$ can be done in polynomial time.
\end{cor}
\begin{proof}
By~\cite{Epp-08b}, partial cubes can be recognized and embedded in a hypercube in quadratic time. For a partial cube embedded in a hypercube checking if it is antipodal can be done in linear time by checking if every vertex has its antipode. Note that the convex hull of any subset can be computed in linear time in the number of edges (for instance by using Lemma~\ref{lem:restriction} for a graph embedded in a hypercube) and checking if a convex subgraph is gated is linear by Lemma \ref{lem:dictionarygated}.
 
 We proceed by designing recognition algorithms using the (i)$\Leftrightarrow$(ii) part of the respective characterizations. We start by   computing  $\conv(u,v)$ for all pairs of vertices $u,v$ and storing them. We check if $\conv(u,v)$ is antipodal, and if so we check if it is gated and if it is isomorphic to a hypercube, i.e.~$|\conv(u,v)|=2^{d(u,v)}$. If all the antipodal graphs obtained in this way are gated, then $G\in\topegraphs$, otherwise we do not proceed. If $G$ is among the antipodal subgraphs, then $\topegraphsOM$. Moreover, if all the antipodal subgraphs are isomorphic to hypercubes,  then $G\in\topegraphsLOP$.
 
We continue by checking for each $\conv(u,v)$ if it is an affine subgraphs. 
 For each pair $u',v'\in \conv(u,v)$ such that $|\conv(u',v')|<|\conv(u,v)|$ we store the pair in $NA(u,v)$, and we search for a pair $w,-_{\conv(u,v)}w\in\conv(u,v)$ such that the set of $\Theta$-classes on a shortest $(u',w)$-path and on a shortest $(v',-_{\conv(u,v)}w)$-path are disjoint. Note that the convex hulls are already computed. If this is the case for all such $u',v'$, we store $\conv(u,v)$ as an affine subgraph. The correctness follows from Proposition \ref{prop:affinepc}.
 
Now we check if the whole graph is affine, in this case say $G=\conv(u,v)$.
Then for every affine subgraph $\conv(u',v')$ and vertex $w \in \conv(u',v')$, we check if the pair $w, -_{\conv(u',v')}w$ is a pair in $NA(u',v')$ if and only if $w, -_Gw$ is a pair in $NA(u,v)$. If this is the case, $\conv(u',v')$ is a conformal subgraph and we check if it is gated. Finally, if all conformal subgraphs are gated, then $G\in\topegraphsAOM$.
\end{proof}

A partial cube is called \emph{Pasch}, their class is denoted by $\mathcal{S}_4$, if any two disjoint convex subgraphs lie in two disjoint halfspaces. We confirm  a conjecture from~\cite{Che-16}:
\begin{cor}\label{cor:S4}
 The class $\mathcal{S}_4$ of Pasch graphs is a subclass of $\topegraphs$.
\end{cor}
\begin{proof}
 Correcting the list given in~\cite{Che-86,Che-94}, in~\cite{Che-16} the list of excluded pc-minors of Pasch graphs has been given, see Figure~\ref{fig:forbidden3}. With this it is easy to verify that $\mathcal{S}_4\subset\mathcal{F}(\Q)$. Another way of seeing the inclusion is using~\cite[Theorem A]{Che-16}, which confirms that the convex closure of any isometric cycle in a Pasch graph is gated. By Lemma~\ref{lem:antipodal_cycle} in particular we have that all antipodal subgraphs are gated and the claim follows from Theorem~\ref{thm:main}.
\end{proof}

\begin{figure}[ht]
\begin{center}
\includegraphics[width = .8\textwidth]{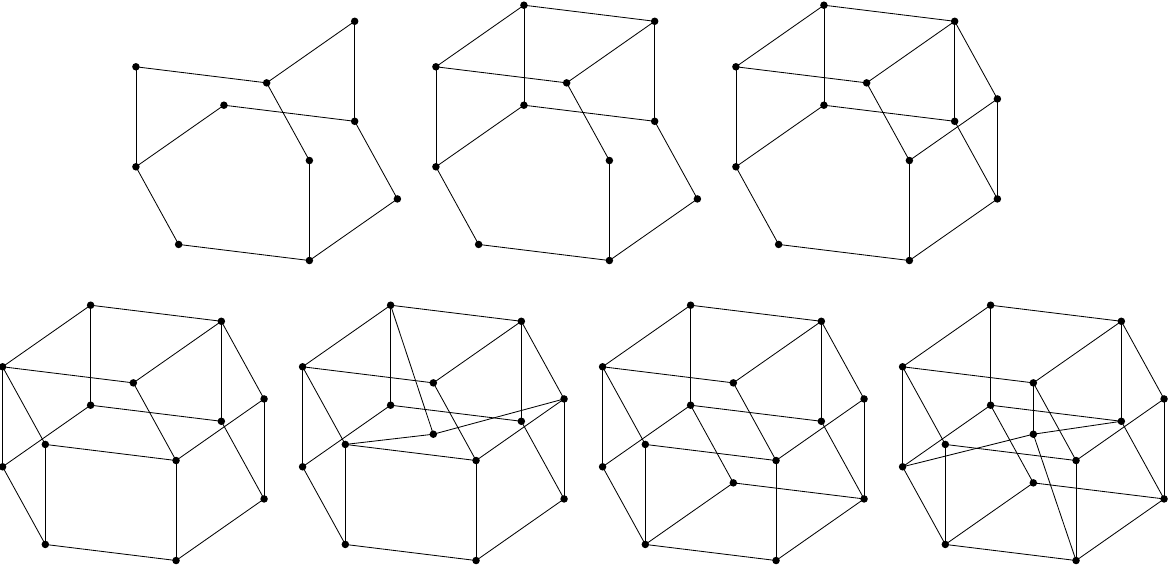}
\caption{The set of minimal forbidden pc-minors of ${\mathcal S}_4$.}
\label{fig:forbidden3}
\end{center}
\end{figure}

Note that together with a recent paper~\cite{Pol-17}, Corollary~\ref{cor:S4} implies that \emph{netlike} partial cubes are tope graphs of COMs. Moreover, it provides an alternative proof for the fact that hypercellular graphs are tope graphs of COMs, see~\cite{Che-16}, and therefore also median graphs, and bipartite cellular graphs.

%


\section{Problems}\label{sec:conclusions}
This final section resumes the central open problems that either result from our paper or might be attacked from a new perspective using our results.

Let us start with the following:
\begin{rem}
Any partial cube is a convex subgraph of an antipodal partial cube.
\end{rem}
To see this, consider an arbitrary partial cube $G$ embedded in a hypercube $Q_n$. Construct an induced subgraph $A_G$ of the hypercube $Q_{n+3}$ where the convex subgraph with the last three coordinates equal to 1 is $G$, the convex subgraph with the last three coordinates equal to $-1$ is $-G$ and the other convex subgraphs given by fixing the last three coordinates are isomorphic to $Q_n$. It is easily checked that $A_G$ is an antipodal partial cube with $G$ a convex subgraph of it. Nevertheless even if $G\in\topegraphs$, $A_G$ is not necessary in $\topegraphs$, e.g., it is easy to check that $A_{C_6}$ has a $Q^{*-}_{4}$-minor. Thus, let us restate a conjecture from~\cite[Conjecture 1]{Ban-15} in terms of tope graphs.
\begin{conj}
Every tope graph of a COM is a convex subgraph of a tope graph of an OM. 
\end{conj}
Note that an affirmative answer to this question would immediately give a topological representation theorem for COMs.

A question arising from Corollaries~\ref{cor:OM} and~\ref{cor:AOM} is based on the observation that $\topegraphsAOM$ and $\topegraphsOM$ are closed under contraction.
\begin{prob}
 Find the list of minimal excluded affine and antipodal contraction-minors for the classes $\topegraphsAOM$ and $\topegraphsOM$, respectively.
\end{prob}
Using Proposition~\ref{prop:affinepc} it is easy to see, that the affine partial cubes in $\Q$ are exactly the graphs of the form $Q^{--}_{n}(m)$. So this gives a family of excluded contraction minors for $\topegraphsAOM$. However, it is not complete since some graphs of this family, e.g. the one arising from $Q^{--}_{4}(1)$, have halfspaces that are COMs which therefore have to be excluded as well. Moreover, the antipodal subgraphs obtained as in the proof of Proposition~\ref{prop:affinepc} from the affine graphs $Q^{--}_{n}(m)$ give a family of excluded contraction minors for $\topegraphsOM$. In particular, this implies that the only non-matroidal antipodal graphs with at most five $\Theta$-classes are the ones coming from the four graphs $Q^{--}_{4}(1), \ldots, Q^{--}_{4}(4)$ -- a result originally due to Handa~\cite{Han-93}. However, also this family is incomplete, since the graph $A_{C_6}$ constructed above has no halfspace in $\Q$.

\bigskip

Several characterizations of planar partial cubes are known~\cite{Alb-16,Des-16}, where the latter corrects a characterization from~\cite{Pet-08}. 
A particular consequence of Corollary~\ref{cor:boundedrank} is that tope graphs of OMs of rank at most three are characterized by the finite list of excluded pc-minors $\Q_3$. By a result of Fukuda and Handa~\cite{Fuk-93} this graph class coincides with the class of antipodal planar partial cubes. Since planar partial cubes are closed under pc-minors, we wonder about an extension of this result to general planar partial cubes:

\begin{prob}
 Find the list of minimal excluded pc-minors for the class of planar partial cubes.
\end{prob}

It is easy to see, that any answer here will be an infinite list, since a (strict) subfamily is given by the set $\{G_n\square K_2\mid n\geq 3\}$, where $G_n$ denotes the \emph{gear graph} (also known as \emph{cogwheel}) on $2n+1$ vertices, see Figure~\ref{fig:obstr}. 

 \begin{figure}[htb]
  \centering
  \includegraphics{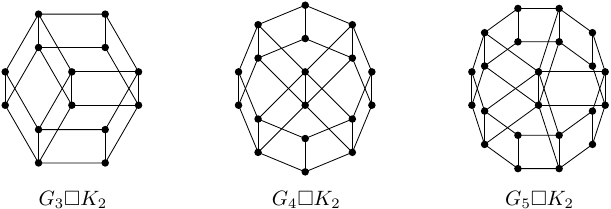}
  \caption{The first three members of an infinite family of minimal obstructions for planar partial cubes.}
  \label{fig:obstr}
 \end{figure}

This in particular shows, that if a pc-minor closed class is contained in an (ordinary) minor closed graph class it can still have an infinite list of excluded minors. A tempting probably quite hard problem is the following:

\begin{prob}
 Give a criterion for when a pc-minor closed class has a finite list of excluded pc-minors.
\end{prob}

%
%
%
%

%

\subsubsection*{Acknowledgments} We wish to thank Emanuele Delucchi and Yida Zhu for discussions on AOMs, Ilda da Silva for insights on OMs, Hans-J\"urgen Bandelt and Victor Chepoi for several fruitful discussions on COMs and their tope graphs, and Matja\v{z} Kov\v{s}e for being part of the very first sessions on tope graphs. The first author was supported by grants ANR-16-CE40-0009-01, ANR-17-CE40-0015, and ANR-17-CE40-0018, the second author by grants P1-0297 and J1-9109. Finally, we wish to thank the referees for comments that clearly improved the quality of this paper.

\bibliographystyle{my-siam}


\begin{thebibliography}{10}
{\small
\bibitem{Alb-16}
{\sc M.~{Albenque} and K.~{Knauer}}, {\em {Convexity in partial cubes: the hull
  number.}}, {Discrete Math.}, 339 (2016), pp.~866--876.

\bibitem{Ban-89}
{\sc H.-J. {Bandelt}}, {\em {Graphs with intrinsic $S\sb 3$ convexities.}}, {J.
  Graph Theory}, 13 (1989), pp.~215--227.

\bibitem{Ban-15}
{\sc H.-J. {Bandelt}, V.~{Chepoi}, and K.~{Knauer}}, {\em {COMs: complexes of
  oriented matroids.}}, {J. Comb. Theory, Ser. A}, 156 (2018), pp.~195--237.

\bibitem{Bau-16}
{\sc A.~{Baum} and Y.~{Zhu}}, {\em {The axiomatization of affine oriented
  matroids reassessed.}}, {J. Geom.}, 109 (2018), p.~21.

\bibitem{Ber-88}
{\sc A.~{Berman} and A.~{Kotzig}}, {\em {Cross-cloning and antipodal graphs.}},
  {Discrete Math.}, 69 (1988), pp.~107--114.

\bibitem{bj}
{\sc A.~Bj{\"o}rner}, {\em Topological methods}, in Handbook of combinatorics,
  Vol.\ 1,\ 2, Elsevier, 1995, pp.~1819--1872.

\bibitem{Bjo-90}
{\sc A.~{Bj\"orner}, P.~H. {Edelman}, and G.~M. {Ziegler}}, {\em {Hyperplane
  arrangements with a lattice of regions.}}, {Discrete Comput. Geom.}, 5
  (1990), pp.~263--288.

\bibitem{bjvestwhzi-93}
{\sc A.~Bj{\"o}rner, M.~Las~Vergnas, B.~Sturmfels, N.~White, and G.~M.
  Ziegler}, {\em Oriented matroids}, vol.~46 of Encyclopedia of Mathematics and
  its Applications, Cambridge University Press, Cambridge, second~ed., 1999.

\bibitem{Che-86}
{\sc V.~Chepoi}, {\em $d$-Convex sets in graphs}, Dissertation (Russian),
  Moldova State University, Chi\c{s}in\v{a}u, 1986.

\bibitem{Che-94}
{\sc V.~{Chepoi}}, {\em {Separation of two convex sets in convexity
  structures.}}, {J. Geom.}, 50 (1994), pp.~30--51.

\bibitem{Che-16}
{\sc V.~{Chepoi}, K.~{Knauer}, and T.~{Marc}}, {\em {Partial cubes without
  $Q_3^-$ minors}}, ArXiv,  (2016).

\bibitem{Che-88}
{\sc V.~{Chepoj}}, {\em {Isometric subgraphs of Hamming graphs and
  $d$-convexity.}}, {Cybernetics}, 24 (1988), pp.~6--11.

\bibitem{Cor-85}
{\sc R.~Cordovil}, {\em A combinatorial perspective on the non-{R}adon
  partitions}, J. Combin. Theory Ser. A, 38 (1985), pp.~38--47.

\bibitem{daS-95}
{\sc I.~P.~F. da~Silva}, {\em Axioms for maximal vectors of an oriented
  matroid: a combinatorial characterization of the regions determined by an
  arrangement of pseudohyperplanes}, European J. Combin., 16 (1995),
  pp.~125--145.

\bibitem{Del-17}
{\sc E.~Delucchi and K.~Knauer}, {\em Finitary affine oriented matroids}.
\newblock (in preparation).

\bibitem{Des-16}
{\sc R.~{Desgranges} and K.~{Knauer}}, {\em {A correction of a characterization
  of planar partial cubes.}}, {Discrete Math.}, 340 (2017), pp.~1151--1153.

\bibitem{Djo-73}
{\sc D.~{\v{Z}}. Djokovi{\'c}}, {\em Distance-preserving subgraphs of
  hypercubes}, Journal of Combinatorial Theory, Series B, 14 (1973),
  pp.~263--267.

\bibitem{Dre-87}
{\sc A.~W.~M. {Dress} and R.~{Scharlau}}, {\em {Gated sets in metric spaces.}},
  {Aequationes Math.}, 34 (1987), pp.~112--120.

\bibitem{Epp-08b}
{\sc D.~{Eppstein}}, {\em {Recognizing partial cubes in quadratic time.}}, in
  {Proceedings of the nineteenth annual ACM-SIAM symposium on discrete
  algorithms, SODA 2008}, 2008, pp.~1258--1266.

\bibitem{Epp-09}
{\sc D.~Eppstein}, {\em Isometric diamond subgraphs}, in Graph Drawing,
  vol.~5417 of Lecture Notes in Computer Science, Springer Berlin Heidelberg,
  2009, pp.~384--389.

\bibitem{Epp-08}
{\sc D.~Eppstein, J.-C. Falmagne, and S.~Ovchinnikov}, {\em Media theory},
  Springer-Verlag, Berlin, 2008.
\newblock Interdisciplinary applied mathematics.

\bibitem{Fuk-04}
{\sc K.~Fukuda}, {\em Lecture Notes on Oriented Matroids and Geometric
  Computation}, ETH Z{\"u}rich, 2004.

\bibitem{Fuk-93}
{\sc K.~Fukuda and K.~Handa}, {\em Antipodal graphs and oriented matroids},
  Discrete Math., 111 (1993), pp.~245--256.
\newblock Graph theory and combinatorics (Marseille-Luminy, 1990).

\bibitem{Fuk-91}
{\sc K.~{Fukuda}, S.~{Saito}, and A.~{Tamura}}, {\em {Combinatorial face
  enumeration in arrangements and oriented matroids.}}, {Discrete Appl. Math.},
  31 (1991), pp.~141--149.

\bibitem{Gra-71}
{\sc R.~L. Graham and H.~O. Pollak}, {\em On the addressing problem for loop
  switching}, Bell System Tech. J., 50 (1971), pp.~2495--2519.

\bibitem{ha-90}
{\sc K.~Handa}, {\em A characterization of oriented matroids in terms of
  topes}, European J. Combin., 11 (1990), pp.~41--45.

\bibitem{Han-93}
{\sc K.~Handa}, {\em Topes of oriented matroids and related structures}, Publ.
  Res. Inst. Math. Sci., 29 (1993), pp.~235--266.

\bibitem{Kar-92}
{\sc J.~Karlander}, {\em A characterization of affine sign vector systems}, PhD
  Thesis, Kungliga Tekniska H{\"o}gskolan Stockholm, 1992.

\bibitem{Kla-12}
{\sc S.~Klav{\v{z}}ar and S.~Shpectorov}, {\em Convex excess in partial cubes},
  J. Graph Theory, 69 (2012), pp.~356--369.

\bibitem{Law-83}
{\sc J.~Lawrence}, {\em Lopsided sets and orthant-intersection by convex sets},
  Pacific J. Math., 104 (1983), pp.~155--173.

\bibitem{Mar-16}
{\sc T.~{Marc}}, {\em {There are no finite partial cubes of girth more than 6
  and minimum degree at least 3.}}, {Eur. J. Comb.}, 55 (2016), pp.~62--72.

\bibitem{Ovc-11}
{\sc S.~{Ovchinnikov}}, {\em {Graphs and cubes.}}, Berlin: Springer, 2011.

\bibitem{Pet-08}
{\sc I.~{Peterin}}, {\em {A characterization of planar partial cubes.}},
  {Discrete Math.}, 308 (2008), pp.~6596--6600.

\bibitem{Pol-17}
{\sc N.~{Polat}}, {\em {On bipartite graphs whose interval space is a closed
  join space.}}, {J. Geom.}, 108 (2017), pp.~719--741.

\bibitem{Vap-15}
{\sc V.~N. {Vapnik} and A.~Y. {Chervonenkis}}, {\em {On the uniform convergence
  of relative frequencies of events to their probabilities.}}, in {Measures of
  complexity. Festschrift for Alexey Chervonenkis}, Cham: Springer, 2015,
  pp.~11--30.

\bibitem{Win-84}
{\sc P.~M. {Winkler}}, {\em {Isometric embedding in products of complete
  graphs.}}, {Discrete Appl. Math.}, 7 (1984), pp.~221--225.
}
\end{thebibliography}
\end{document}